\documentclass[twoside,leqno]{article}

\usepackage[letterpaper]{geometry}

\usepackage{siamproceedings}
\usepackage{mathtools}

\usepackage[T1]{fontenc}
\usepackage{amsfonts}
\usepackage{amsmath, amssymb}
\usepackage{graphicx}
\usepackage{epstopdf}
\usepackage[shortlabels]{enumitem}
\setlist[enumerate]{(\arabic*)}

\usepackage{algorithmic}
\ifpdf
  \DeclareGraphicsExtensions{.eps,.pdf,.png,.jpg}
\else
  \DeclareGraphicsExtensions{.eps}
\fi

\newsiamremark{remark}{Remark}
\newsiamremark{rem}{Remark}
\newsiamremark{hypothesis}{Hypothesis}
\crefname{hypothesis}{Hypothesis}{Hypotheses}
\newsiamthm{claim}{Claim}
\newsiamthm{lem}{Lemma}
\newsiamthm{defn}{Definition}
\newsiamthm{thm}{Theorem}
\newsiamthm{conj}{Conjecture}
\newsiamthm{cor}{Corollary}
\newsiamthm{prop}{Proposition}

\usepackage{amsopn}

\usepackage{tikz} 
\usepackage{tikz-cd}
\usetikzlibrary{arrows.meta,bending,calc,math,decorations.markings,hobby,patterns,plotmarks}

\newcommand{\mf}{\mathfrak}
\newcommand{\DSen}{\mathbf{D}_{\operatorname{Sen}}}
 
 \newcommand{\WM}{{}^{M}W}

 \def\Sh{\mathrm{Sh}}

 \def\Stab{\mathrm{Stab}}
\def\sm{\mathrm{sm}}

\def\nfs{f\kern-0.06em{s}}
\def\Sh{\mathrm{Sh}}

\newcommand{\oscr}{\mathcal{O}}
\newcommand{\Loc}{\operatorname{Loc}}

\newcommand{\colim}{\operatorname{colim}}
\newcommand{\Sen}{\operatorname{Sen}}

\newcommand{\Lmw}{L(\mathfrak{m}_w)}

\newcommand{\ocal}{{\mathcal O}}

\newcommand{\an}{\operatorname{an}}

\def\A{\mathbf A}

\def\C{\mathbf C}
\def\F{\mathbf F}

\def\Q{\mathbf{Q}}
\def\R{\mathbf{R}}

\def\Z{\mathbf{Z}}
\def\Fbar{\overline{\F}}

\def\Qbar{\overline{\Q}}

\def\m{\mathfrak m}

\def\Mod{\mathrm{Mod}}

\def\id{\mathrm{id}}

\def\alg{\mathrm{alg}}

\def\sm{\mathrm{sm}}

\def\SL{\mathrm{SL}}

\def\ord{\mathrm{ord}}

\def\GL{\operatorname{GL}}

\def\Gal{\mathrm{Gal}}

\def\Sym{\mathrm{Sym}}

\def\End{\mathrm{End}}

\def\RHom{\mathop{\mathrm{RHom}}\nolimits}

\def\Frob{\mathop{\mathrm{Frob}}\nolimits}

\def\rhobar{\overline{\rho}}

\def\dR{\mathrm{dR}}

\def\m{\mathfrak{m}}

\newcommand{\la}{\mathrm{la}}

\newcommand{\into}{\hookrightarrow}

\newcommand{\To}{\longrightarrow}
\newcommand{\isoto}{\stackrel{\sim}{\To}}

\newlength{\ownl}

\newcommand{\Id}{{\operatorname{Id}}}

\newcommand{\Lie}{{\operatorname{Lie}\,}}

\newcommand{\Res}{{\operatorname{Res}}}

\newcommand{\GSp}{\operatorname{GSp}}

\newcommand{\PGL}{\operatorname{PGL}}

\newcommand{\Sp}{\operatorname{Sp}}

\newcommand{\HH}{H}

\newcommand{\cC}{\mathcal{C}}

\newcommand{\cF}{\mathcal{F}}

\newcommand{\cO}{\mathcal{O}}

\newcommand{\qq}{\Q}

\newcommand{\varepsilonbar  }{\overline{\varepsilon}}

 \newcommand{\varrhobar   }{\overline{\varrho}}

 \newcommand{\cL}{\mathcal{L}}

\newcommand{\HT}{\operatorname{HT}}

 \newcommand{\Qp}{{\Q_p}}

\newcommand{\Zp}{{\Z_p}}
\newcommand{\Ql}{\Q_l} \newcommand{\Cp}{\C_p}
\newcommand{\Qpbar}{{\overline{\Q}_p}}

\newcommand{\Zpbar}{{\overline{\Z}_p}}

\newcommand{\Qlbar}{\overline{\Q}_{\ell}}
\newcommand{\Fpbar}{{\overline{\F}_p}}

\newcommand{\Flbar}{\overline{\F}_l}

\newcommand{\Jac}{\operatorname{Jac}}

\begin{document}

\title{Modularity theorems for abelian surfaces}
\author{Toby Gee} 
\maketitle

\renewcommand{\thefootnote}{\arabic{footnote}} 
\begin{abstract}This is a brief account of my results with George Boxer, Frank Calegari and Vincent Pilloni on the (potential) modularity of abelian surfaces.
 \end{abstract}

\section{Introduction}\label{sec:intro}
This article gives an overview of the proofs of the main results of~\cite{BCGP,boxer2025modularitytheoremsabeliansurfaces}.
I have attempted to complement the introductions to those papers, and to concentrate on aspects of the proofs that are not already covered in other survey articles.
Consequently, I devote  little space to innovations in the Taylor--Wiles method (in particular, the Calegari--Geraghty method), for which I refer the reader to Calegari's excellent surveys~\cite{MR4680265,MR4363376}.

We begin by recalling the main theorems of the papers. Our first paper~\cite{BCGP} proves the Hasse--Weil conjecture for abelian surfaces (and genus 2 curves) over totally real fields. To recall what this means, let $X$ be a smooth, projective variety over a number field~$F$  with good reduction outside a finite set of primes~$S$.
Associated to~$X$, one may write down a global Hasse--Weil zeta function:
$$\zeta_X(s) = \prod \frac{1}{1 - N(x)^{-s}},$$
where the product runs over all the closed points~$x$ of some (any) smooth proper integral model~$\mathcal{X}/\cO_F[1/S]$ for~$X$.
(Different choices of~$S$ only change~$\zeta_X(s)$ by a finite number of Euler factors. For curves and abelian varieties there is a natural definition of the Euler factors at all places, and our modularity results are compatible with these factors, but we ignore this point from now on.)
The function~$\zeta_X(s)$ is absolutely convergent for~$\mathrm{Re}(s)
> 1 + \dim X$. We have the following:
\begin{conj}[Hasse--Weil Conjecture, cf.~\cite{Serre1969-1970}, in particular Conj.\ C9] \label{conj:serre} The function~$\zeta_X(s)$ extends to a meromorphic function of~$\C$. There exists a positive real number~$A \in \R^{>0}$,  non-zero rational functions~$P_v(T)$ for~$v|S$, and infinite Gamma factors~$\Gamma_v(s)$ for~$v| \infty$ such that:
$$\xi(s) = \zeta_X(s) \cdot  A^{s/2}  \cdot \prod_{v | \infty} \Gamma_v(s)  \cdot \prod_{v|S} P_v(N(v)^{-s})$$
satisfies the functional equation~$\xi(s) = w \cdot \xi(\dim X+1-s)$ with~$w = \pm 1$.
\end{conj}
(In Serre's formulation of the conjecture, the Gamma factors are also given explicitly in terms of the Archimedean Hodge structures of $X$.)

If~$F = \Q$ and~$X$ is a point, then~$\zeta_X(s)$ is  the Riemann zeta function, 
and Conjecture~\ref{conj:serre} follows from Riemann's functional equation.
By work of Hecke and Brauer, the conjecture is known if~$X$ is zero-dimensional, or if more generally
the Galois representations associated to the~$l$-adic
cohomology of~$X$ are \emph{potentially} abelian (e.g.\ an abelian
variety with CM). 
All subsequent progress on Conjecture~\ref{conj:serre} has been via the Langlands program.
Write~  $\Gal_{\Q}$ for the absolute Galois group of~$\Q$; more generally, for any field~$K$,  we write~$\overline{K}$ for a separable closure, and $\Gal_{K}\coloneqq \Gal(\overline{K}/K)$ for the absolute Galois group.
If~$K$ is a local field, then we write~$I_{K}$ for the inertia subgroup of~$\Gal_{K}$.
Via the Grothendieck--Lefschetz trace formula, one writes for each prime~$\ell$ (at least up to a finite number of Euler factors) \[\zeta_X(s)=\prod_{n=0}^{2\dim
    X}L(H^n(X_{\overline{F}},\Qlbar),s)^{(-1)^n}\]where~ $L(H^n(X_{\overline{F}},\Qlbar),s)$ is the $L$-function of the~$\ell$-adic representation $H^n(X_{\overline{F}},\Qlbar)$ of~$\Gal_{\Q}$.
The Langlands conjectures predict that each $L(H^n(X_{\overline{F}},\Qlbar),s)$ is a product of automorphic  $L$-functions; more precisely, we have $L(H^n(X_{\overline{F}},\Qlbar),s)=\prod_jL(\pi_j,s)$, where each~$\pi_j$ is an automorphic representation of~$\GL_{n_j}(F)$ for some integers~$n_j$ with $\sum_jn_j=\dim_{\Qlbar}H^n(X_{\overline{F}},\Qlbar)$. Since (completed) automorphic $L$-functions have meromorphic (usually holomorphic) continuations and functional equations, this prediction implies the Hasse--Weil conjecture.

\begin{defn}\label{defn:modular}
  If~$X/F$ is a curve or an abelian variety, we write $L(X,s)\coloneqq L(H^1(X_{\overline{F}},\Qlbar),s)$, and we say that~$X$ is \emph{modular} if $L(X,s)$ is a product of automorphic $L$-functions.  We say that~$X$ is \emph{potentially modular} if there is a finite extension $F'/F$ such that $X_{F'}$ is modular.
  Similarly, we say that a representation~$\rho:\Gal_F\to\GL_n(\Qlbar)$ is modular if~$L(\rho,s)$ is a product of automorphic $L$-functions, and we say that a representation~$\rhobar:\Gal_F\to\GL_n(\Flbar)$ is modular if it is the reduction modulo~$\ell$ of a modular Galois representation.
\end{defn}
\begin{rem}
  The notion of potential modularity was introduced by Taylor~\cite{MR1954941}, who observed that Brauer's methods apply in this setting, so that if~$X$ is potentially modular, then ~$L(X,s)$ has the expected meromorphic continuation and functional equations.
However one cannot in general say anything about the poles (or lack thereof) of this meromorphic continuation.
\end{rem}

If~$X$ is a curve of genus zero, then (up to bad Euler factors) $\zeta_X(s) = \zeta_F(s) \zeta_F(s-1)$, and Conjecture~\ref{conj:serre} follows immediately. The fundamental work of Wiles~\cite{MR1333035,MR1333036} and the subsequent
work of Breuil, Conrad, Diamond, and Taylor~\cite{CDT,BCDT} proved Conjecture~\ref{conj:serre} for curves~$X/\Q$ of genus one,
since if we write~ $E=\Jac(X)$, then $\zeta_X(s) = \zeta(s) \zeta(s-1)/L(E,s)$ so the modularity of~$E$ implies the holomorphy and functional
equation for~$L(E,s)$. More generally, Taylor's potential modularity results~\cite{MR1954941}  prove Conjecture~\ref{conj:serre} for curves~$X/F$ of genus one over any totally real field.

The methods used in these papers have been vastly generalized over the past~30 years
due to the enormous efforts of many people, and as a consequence one knows for example that if~$F$ is totally real and~$X$ is such that   the Hodge numbers~$h^{p,q} = \dim H^{p,q}_{\dR}(X) = \dim
H^q(X,\Omega^p)$ of~$X$ are at most~$1$ for all~$p$ and~$q$ with $p+q=n$, then $L(H^n(X_{\overline{F}},\Qlbar),s)$ has the expected meromorphic continuation and functional equation (see~\cite[Cor.\ B]{MR3314824}).

Unfortunately, this multiplicity one condition on Hodge numbers is fundamental to the  original Taylor--Wiles method, and there is a paucity of natural geometric examples satisfying this condition. In particular, it fails for curves of genus~$g>1$ and for abelian varieties of dimension~$g>1$, where ~$h^{1,0}=h^{0,1} = g$.   
The main theorem of~\cite{BCGP} is the following.
\begin{theorem}\label{theorem:puppy} Let~$X$ be either a genus two curve or an abelian surface over a totally real field~$F$. Then~$X$ is potentially modular, and  Conjecture~\ref{conj:serre}
holds for~$X$.
\end{theorem}
(The deduction of  Conjecture~\ref{conj:serre} from potential modularity is straightforward in this case, using that the cohomology of~$X$ is given by the wedge powers of the cohomology in degree~$1$, and  known Langlands functoriality results for wedge powers.)

While Theorem~\ref{theorem:puppy} resolves the Hasse--Weil conjecture for abelian surfaces~$A/\Q$, for many purposes (e.g.\ applications to the Birch--Swinnerton-Dyer conjecture) one wishes to know modularity.

Modularity is known in some cases, using the results recalled above (in particular, the modularity of elliptic curves).
More precisely, by~\cite[Thm.\ 10.2.1]{boxer2025modularitytheoremsabeliansurfaces}, it is known unless $A/\Q$ is ``challenging'' in the sense of~\cite[\S9.2]{BCGP}, which means that either
\begin{enumerate}
\item\label{item:challenging-1} $\End(A_{\Qbar}) = \Z$, or
\item \label{item:challenging-2}  there exists a quadratic field~$K/\Q$ so that~$\End(A) = \Z$ but $\End(A_K) \otimes \Q$ is either~$\Q \oplus \Q$ or a real quadratic field.
\end{enumerate}
A natural source of abelian surfaces of type~\ref{item:challenging-2} are those of the form~$\Res_{K/\Q}(E)$ for a non-CM elliptic curve~$E$ which is not isogenous to its~$\Gal(K/\Q)$-conjugate. In this case the modularity of~$A$ would follow from the modularity of~$E$.
If~$K/\Q$ is real quadratic, then $E$ is modular by Freitas--Le Hung--Siksek~\cite{FLS}, while if~$K$ is imaginary quadratic, then the modularity of ~$E$ is known in many cases by work of Caraiani--Newton~\cite{caraiani-newton}.
On the other hand, the endomorphism algebra $\End(A_K) \otimes \Q$ could also be a real quadratic field~$E$ rather than~$\Q \times \Q$, in which case~$A/K$ will be a simple abelian surface of~$\GL_2$-type, and the modularity of such abelian surfaces remains open in general even for real quadratic fields~$K$.

In view of this, we concentrate from now on the ``typical'' case that~$\End_{\Qbar}(A) = \Z$, where one has the more precise expectation that~$L(A,s)=L(\pi,s)$ for some cuspidal automorphic representation~$\pi$ of~$\GSp_4 /\Q$, and Brumer and Kramer~\cite{MR3165645} formulated the \emph{paramodular conjecture}, which gives a precise prescription for the ``optimal'' level structure for an automorphic form corresponding to a given abelian surface; in particular, this in principle reduces the
conjecture for a given~$A$ to an explicit computation of a (finite-dimensional) space of Siegel modular forms. Using the Faltings--Serre method, and  elaborate explicit computations of low weight Siegel modular forms,
developed in part by Poor and
Yuen~\cite{MR3315514,MR3713095,MR3498287}, the modularity of (finitely many, up
to twist) abelian surfaces~$A$ with~$\End_{\Qbar}(A) = \Z$ was established in the papers \cite{MR3315514,BPVY,Berger}.

 If~$A/F$ is an abelian surface, we write~$\rho_{A,p}$ for the Galois representation associated to~\(H^1(A_{\overline{F}},\Z_p)\), which we often think of as a representation
\[\rho_{A,p}: \Gal_{F} \to \GSp_4(\Q_p)\]
with multiplier given by the inverse cyclotomic character~$\varepsilon^{-1}$.
By definition, $A$ is modular if and only if~$\rho_{A,p}$ is modular for some~$p$ (equivalently, for all~$p$).
We also let \(\rhobar_{A,p}\) denote the Galois
 representation associated to \(H^1(A_{\overline{F}},\F_p)\). If \(A\) admits a principal polarization
 of degree prime to \(p\), then we can and do think of \(\rhobar_{A,p}\) as a representation
 \[\rhobar_{A,p}: \Gal_{F} \to \GSp_4(\Fbar_p).\]
 
The main theorem of \cite{boxer2025modularitytheoremsabeliansurfaces} is as follows.

\begin{thm} \label{first}
Let~$A/\Q$ be an abelian surface with a polarization of degree prime to~$3$. Suppose 
the following hold:
\begin{enumerate}
\item \label{surjective} The mod~$3$ representation
$$\rhobar_{A,3}: \Gal_{\Q} \rightarrow \GSp_4(\F_3)$$
is surjective.
\item \label{conditionattwo} The representation $\rhobar_{A,3}|_{\Gal_{{\Q_2}}}$ is unramified, and the characteristic polynomial 
of $\rhobar_{A,3}(\Frob_2)$ is not $(x^2\pm x+2)^2$. 
\item \label{conditionatthree} $A$ has good ordinary reduction at~$3$, and
 the characteristic polynomial
of ~$\Frob_3$ does not have repeated roots.
\end{enumerate}
Then~$A$ is modular.

More precisely, there exists a cuspidal automorphic representation~$\pi$
of~$\GL_4/\Q$ (the transfer of a cuspidal
automorphic representation of~$\GSp_4/\Q$ of weight~$2$)
such that~$L(s,H^1(A)) = L(s,\pi)$. Consequently, $L(s,H^1(A))$ has a holomorphic continuation to~$\C$ and satisfies the expected functional equation.
\end{thm}

\begin{rem}We claim that Theorem~\ref{first} applies to a positive proportion of abelian surfaces over~$\Q$, counted in any reasonable sense.
As one justification of this, suppose that one samples genus two curves
$$X:y^2 + h(x) y = f(x)$$
with~$h(x), f(x) \in \Z[x]$ of degrees~$\le 3$ and~$\le 6$ in any way in which the distributions modulo~$2$ and~$3$ are equidistributed, and considers those curves~$X$ with the following properties:
\begin{enumerate}
\item $\rhobar_{\Jac{X},3}: \Gal_{\Q} \rightarrow \GSp_4(\F_3)$
is surjective,
\item $X$ has good reduction at~$2$,
\item $X$ has good ordinary reduction at~$3$,
\item $\rhobar_{\Jac(X),3}(\Frob_2)$ does not have characteristic polynomial  $(x^2\pm x+2)^2$,
\item The characteristic polynomial of~$\Frob_3$ has distinct eigenvalues.
\end{enumerate}
 Theorem~\ref{first} proves the modularity of~$\Jac(X)$ for any such~$X$, and one can check~\cite[\S 10.1]{boxer2025modularitytheoremsabeliansurfaces} that these~$X$ form a subset of density
$\frac{5551}{46656} = 0.1189\ldots $ Another point of comparison is with the curves in the database~\cite{LMFDB} (see also~\cite{database}). This contains~$63107$ genus two curves~$X/\Q$  with $\End \Jac(X)_{\Qbar} = \Z$, and Theorem~\ref{first} applies to~$11384$ of them.
\end{rem}

\section{The 2--3 switch}
\label{sec:2-3-switch}
The proof of Theorem~\ref{first} follows Wiles's strategy for proving the modularity of semistable elliptic curves, and in particular, we make use of an analogue of 
the~$3$-$5$ switch used by Wiles~\cite{MR1333035}
 to prove residual modularity. That switch exploited the rationality of certain twists of the modular curve~$X(5)/\Q$. In our case, we use a rational moduli
 space of abelian surfaces to carry out a~$2$-$3$ switch.

In outline, the $2$-$3$ switch proving Theorem~\ref{first} divides into three steps as follows.
\begin{enumerate}[Step \arabic*]    \item\label{item:B_2 modular} Show that~$\rhobar_{B,2}$ is modular for many abelian surfaces~$B/\Q$. (See Lemma~\ref{lem: A5 solvable potential modularity B[2]}.)   \item \label{item:A-gives-B}  Show that for any abelian surface~$A$ as in Theorem~\ref{first}, there exists~$B$ as in \eqref{item:B_2 modular}  with $\rhobar_{B,3}\cong \rhobar_{A,3}$. (See Lemma~\ref{switching}.)
      \item  \label{2-3-switch-modularity-lifting-step}
      Prove the following (imprecisely stated)       modularity lifting theorem, which applies in particular to the representations~$\rho_{B,2}$ and~$\rho_{A,3}$ for~$A,B$ as in the previous two steps:
      \begin{thm}\label{idealthm:modularity lifting} Suppose that $\rho:\Gal_{\Q}\to\GSp_4 (\Zpbar)$ is unramified at all but finitely many primes and de Rham at~$p$, and:
        \begin{enumerate}[(\roman*)]
          \item $\rhobar:\Gal_{\Q}\to\GSp_4(\overline{\F}_{p})$ is modular.
      \item \label{hyp:big-image} $\rhobar(\Gal_{\Q})$ is large.
      \item $\rho$ is pure. 
      \item \label{hyp:ordinary-p-distinguished}$\rho|_{\Gal_{\Qp}}$ is ordinary,  $p$-distinguished,  and has Hodge--Tate weights $0,0,1,1$.
            \end{enumerate}

      Then~$\rho$ is modular.      
    \end{thm}
                  \end{enumerate}

  \begin{proof}[Proof of Theorem~\ref{first}, given these steps]
 Suppose that~$A$ satisfies the hypotheses of Theorem~\ref{first}, and let~$B$ be as in \ref{item:A-gives-B}.
Then ~$\rhobar_{B,2}$ is modular by~\ref{item:B_2 modular}, so that~$\rho_{B,2}$ is modular by Theorem~\ref{idealthm:modularity lifting}.
Equivalently, $\rho_{B,3}$ is modular, so that~$\rhobar_{B,3}$ is modular.
Since $\rhobar_{A,3}\cong\rhobar_{B,3}$ by assumption, we can apply Theorem~\ref{idealthm:modularity lifting} to deduce that~$\rho_{A,3}$ is modular, as required.
 \end{proof}

 \begin{rem}
   Hypothesis~\ref{hyp:big-image} of Theorem~\ref{idealthm:modularity lifting} is responsible for assumption~\ref{surjective} in Theorem~\ref{first}, while the more serious hypothesis~\ref{hyp:ordinary-p-distinguished} corresponds to assumption~\ref{conditionatthree} there (and is also responsible for~\ref{conditionattwo}).
 \end{rem}

 Our supply of abelian surfaces~$B/\Q$ for~\ref{item:B_2 modular} will be certain Jacobians~$B=\Jac(X)$, where~$X/\Q$ is a  genus two curve. Let~$r_i$ for~$i = 1,\ldots,6$ be the Weierstrass points of~$X$ over~$\Qbar$;  then the non-zero elements of~$B[2]$ are given by the ~$r_i - r_j$ for $i<j$.
Considering the action of~$\Gal_{\Q}$ on the~$r_i$, one has an identification
$S_6 \isoto \Sp_4(\F_2) = \GSp_4(\F_2)$. There are two conjugacy classes of subgroup~$S_5 \subset S_6$; we denote by~$S_5(b)$ the  standard copy of~$S_5$ in~$S_6$ (and below we write $A_5(b)$ for the copy
of~$A_5$ in ~$S_5(b)$).
Thus~$X$ has a rational Weierstrass point (so that~$X$ can be written in the form $y^2=f(x)$ with~$f$ quintic) if and only if~$\rhobar_{A,2}$ factors through a conjugate of~$S_5(b)$. 

 The following lemma, which exploits some coincidences in the representation theory of~$A_5 $, allows us to find many~$X$ for which we know that~$\rhobar_{B,2}$ is modular.
\begin{lemma}\label{lem: A5 solvable potential modularity B[2]}  Suppose that~$X/\Q$ is a genus two curve with a rational Weierstrass point, and that~$B\coloneqq \Jac(X)$ has semistable ordinary or good ordinary reduction at~$2$.
Suppose also that 
\[\rhobar_{B,2}: \Gal_{\Q} \rightarrow \GSp_4(\F_2) \simeq S_6\] has image~$S_5(b)$, and that the image of complex conjugation has conjugacy class~$(**)(**)$.
Then~$\rhobar_{B,2}$ is modular, arising from an ordinary weight 3 Siegel modular form.
\end{lemma}
\begin{proof}
  If~$F^+$ is the quadratic field given by the kernel of the composite $\Gal_{\Q}\to S_5 (b)\to \Z/2\Z$, then $\rhobar(\Gal_{F^+})= A_5(b)$, and~$F^+$ is real by the assumption on complex conjugation.
Let
$$\varrhobar:\Gal_{F^+} \rightarrow \SL_2(\F_4) \simeq A_5$$
denote the residual $2$-dimensional Galois representation associated to
this~$A_5$-extension. (There are two such representations which are
permuted by the outer automorphism; choose either.) Either by an easy Brauer character
computation, or as a consequence of the Steinberg tensor product
theorem for~$\SL_2(\F_4)$, we have 
\[\rhobar_{B,2}\cong\Sym^3 \varrhobar.\]

By a theorem of Tate
the composite~$\varrhobar: \Gal_{F^{+}} \rightarrow A_5 \hookrightarrow
\PGL_2(\C)$
lifts to an odd representation
$\varrho:\Gal_{F^+}\to
\GL_2(\C)$
with finite image (which will be some central extension of~$A_5$).
By the odd Artin conjecture for~$\GL_2$ (i.e.\ by the main results
of~\cite{MR3581178} or \cite{MR3904451}), $\varrho$ is modular. By Hida theory, it follows that~$\varrhobar$ is modular, coming from an ordinary Hilbert modular eigenform of parallel weight~$2$.
By symmetric cube functoriality~\cite{MR1923967},
 $\rhobar_{B,2}|_{\Gal_{F^+}}$ is modular, arising from an ordinary weight 3 Hilbert--Siegel modular form.
By solvable base change and a standard use of the Khare--Wintenberger method~\cite{MR2480604}, $\rhobar_{B,2}$ itself is modular, arising from an ordinary weight 3 Siegel modular form, as required.
\end{proof}

\ref{item:A-gives-B} is provided by the following lemma (see~\cite[Lem.\ 9.4.2]{boxer2025modularitytheoremsabeliansurfaces}).
\begin{lemma} \label{switching}
  Let~$A/\Q$ be an abelian surface with a polarization of degree prime to~$3$.  Assume that $\rhobar_{A,3}|_{{\Gal_{\Q_{2}}}}$ is unramified, and 
 the characteristic polynomial 
of $\rhobar_{A,3}(\Frob_2)$ is not $(x^2\pm x+2)^2$.

Then there exists a genus two curve~$X/\Q$ with a rational Weierstrass point,
with~$B = \mathrm{Jac}(X)$ having the following properties:
\begin{enumerate}
\item\label{item:B3 cong A3} $\rhobar_{B,3}\cong\rhobar_{A,3}$. \item $B$ has good ordinary  or semistable ordinary reduction at~$2$,
 and is~$2$-distinguished.
\item $B$ has good ordinary reduction at~$3$. 
\item $\End(B_{\Qbar}) = \Z$.
\item The representation
$$\rhobar_{B,2}: \Gal_{\Q} \rightarrow \GSp_4(\F_2)$$
has image $S_5(b)$, and the image of complex conjugation has
conjugacy class~$({*}{*})({*}{*})$.
\end{enumerate}
\end{lemma}
\begin{proof}
  We consider the moduli space~$M(\rhobar_{A,3})/\Q$ of genus-two curves~$X$ equipped with a fixed rational Weierstrass point and a symplectic isomorphism~$\rhobar_{\Jac(X),3}\cong\rhobar_{A,3} $.
By definition, any~$\Q$-point of~$M(\rhobar_{A,3})$ gives $B=\Jac(X)$ satisfying~\eqref{item:B3 cong A3}, and $\rhobar_{B,2}(\Gal_{\Q})\subseteq S_5(b)$ (due to the rational point).

By~ \cite[Lem.\ 10.2.4]{BCGP}, the variety $M(\rhobar_{A,3})/\Q$ is rational. (The proof of this lemma is very similar to that of~\cite[Lem.\ 1.1]{MR1415322}: namely, a Galois cohomology argument reduces to checking that the corresponding moduli space for the representation $1\oplus 1\oplus \varepsilonbar^{-1}\oplus\varepsilonbar^{-1}$ is $\PGL_4(\F_3 ) $-equivariantly rational over~$\Q$, which is known by ~\cite[Theorem~0.0.1]{Weddle}.)
In particular, by weak approximation (combined with Hilbert irreducibility as in \cite[\S 3.4, \S 3.5]{MR2363329}), for any finite set of places~$S$ and any nonempty open subsets $\Omega_{\ell}\subset M(\rhobar_{A,3})(\Ql)$ (in the $\ell$-adic topology) for each~$\ell\in S$, we can find a $\Q$-point of~$M(\rhobar_{A,3})$ lying in~$\Omega_{\ell}$ for each~$\ell\in S$, for which $\rhobar_{B,2}(\Gal_{\Q})=S_5 (b)$ (which in turn implies that $\End(B_{\Qbar}) = \Z$).
We then take $S=\{2,3,\infty\}$; the hypothesis that the characteristic polynomial 
of $\rhobar_{A,3}(\Frob_2)$ is not $(x^2\pm x+2)^2$ is used to guarantee the existence of suitable points over~$\Q_2 $, by writing down appropriate abelian surfaces over~$\Q_2 $ for each of the other possibilities for this characteristic polynomial.
\end{proof}

For Theorem~\ref{theorem:puppy}, we use a variant of the above strategy, where we choose primes~$p,q$ splitting completely in our totally real field~$F$, with the further properties that $A$ admits a polarization of degree prime to~$pq$,  the representations~$\rhobar_{A,p}$ and~$\rhobar_{A,q}$ have large image, and~$A$ has good ordinary reduction at all places dividing~$pq$.
Then we have:
\begin{lem}
  \label{lem:pq-switch} Let~$A/F$ be a challenging abelian
  surface over a totally real field.  Then there is a finite Galois extension of
  totally real fields~$F'/F$ such that~$p$ and~$q$ split completely in~$F'$, and  an abelian surface~$B/F'$ with good ordinary reduction at all places dividing~$pq$, such that 
  ~$\rhobar_{B,p}\cong\rhobar_{A,p}|_{\Gal_{F'}}$, while~$\rhobar_{B,q}$ is modular of parallel weight~$2$, and has large image.
\end{lem}
\begin{proof}
  We consider the moduli space~$Y$ of abelian surfaces~$B$ equipped with isomorphisms $B[p]\isoto A[p]$ and $B[q]\cong \rhobar_q$, for any choice of~$\rhobar_q:\Gal_F\to\GSp_4 (\F_q)$.  There is no reason that~$Y(F)$ should be nonempty, but a theorem of Moret-Bailly~\cite{mb} guarantees that we can find~$F'$ as above so that~$Y(F')$ is nonempty, and satisfies the required local conditions at primes dividing~$pq$.
Choosing~$\rhobar_q$ to be induced from a 2-dimensional representation, we are able to arrange (using the (known) potential modularity of elliptic curves) that after a further extension of~$F'$, $\rhobar_{B,q}$ is modular.
\end{proof}
\begin{rem}
  \label{rem:regular weight or not}As well as the (important!) difference between modularity and potential modularity, another significant difference between the Lemma~\ref{lem:pq-switch} and Lemmas~\ref{lem: A5 solvable potential modularity B[2]} and~\ref{switching} is that Lemma~\ref{lem:pq-switch} shows residual potential modularity in weight~$2$, while Lemma~\ref{lem: A5 solvable potential modularity B[2]} shows residual modularity in weight~$3$. (which is regular).

  The Galois representations associated to Siegel modular eigenforms of weight~$k\ge 2$ have Hodge--Tate weights $0,k-2,k-1,2k-3$ (see Section~\ref{sec:Sen-theory} below for the definition of Hodge--Tate weights), while the Galois representations associated to abelian surfaces have Hodge--Tate weights $0,0,1,1$.
It follows that if an abelian surface is modular, corresponding to a Siegel modular form, that modular form must be of weight~$2$; so when we have residual modularity in weight~$3$, we have to ``change weight'' in some way (see Remark~\ref{rem:history-of-BCGP} below for more discussion of this).
On the other hand, if~$k>2$ then the Hodge--Tate weights $0,k-2,k-1,2k-3$ are pairwise distinct, and we say that~$k$ is a ``regular weight'', while~$k=2$ is an ``irregular weight''.
The irregular weight cases behave quite differently (and are in general much more complicated than the regular weight cases) on both the automorphic and Galois sides of the Langlands correspondences, as we will see below.
\end{rem}

The remainder of this survey is devoted to ~\ref{2-3-switch-modularity-lifting-step} (and its analogue in~\cite{BCGP}). This step divides into two parts. Firstly, we use the Taylor--Wiles method to show that~$\rho$ is $p$-adically modular, in the sense that it contributes to a Hida family of Siegel modular forms. This is relatively standard, although the need to consider the primes~$p=2,3$ causes some pain. Secondly, we prove a classicality criterion for weight~$2$ ordinary Siegel $p$-adic modular
forms. It is in this step that the difference mentioned in Remark~\ref{rem:regular weight or not} becomes significant. In~\cite{BCGP}, the relevant classicality theorem is a ``small slope implies classical'' theorem in the style of Coleman, but this approach does not suffice in the situation considered in ~\cite{boxer2025modularitytheoremsabeliansurfaces}. In this case the proof, building upon work of Rodríguez Camargo~\cite{camargo2022locally}, is a generalisation from~$\GL_2$ to~$\GSp_4$ of a part of Lue Pan's remarkable work~\cite{MR4390302}.
In the remainder of this survey, we concentrate on the classicality theorems,  before briefly returning to the Taylor--Wiles method in Section~\ref{sec:tayl-wiles-patch}.

\begin{rem}\label{rem:history-of-BCGP}
 It was well-known for many years that Theorem~\ref{theorem:puppy} could be deduced from strong enough modularity lifting theorems; we refer the reader to~\cite[\S 6]{MR3966765} and \cite[\S 11.2]{MR4680265} for some of the history.
The paper \cite{BCGP} was posted online at the end of 2018, and we found the strategy outlined above for proving Theorem~\ref{first} in March 2019.
At that time, we imagined that we would need to prove a ``low weight mod~$p$ companion form'' result to show that the weight 3 Siegel modular form in Lemma~\ref{lem: A5 solvable potential modularity B[2]} is congruent to a weight 2 Siegel modular form, and then apply the modularity lifting theorems of~\cite{BCGP}.
We still do not know how to prove such companion form theorems, but the situation changed in 2020 with Lue Pan's paper~\cite{MR4390302}.
By the spring of 2022, we were confident that it was possible to use Pan's techniques to prove an appropriate classicality result, but we had not yet proved  (regular weight) modularity lifting theorems for~$\GSp_4 $ which could applied to the representations~$\rhobar_{B,2}$.
After some false starts, we managed to do this in spring 2023, and we wrote~\cite{boxer2025modularitytheoremsabeliansurfaces} over the following two years.
\end{rem}

\section{Galois representations}\label{sec:Galois-theory}
\subsection{Sen theory}\label{sec:Sen-theory}
Let~$\Qp(\zeta_{p^{\infty}})\coloneqq \cup_n\Qp(\zeta_{p^{n}})$ be the cyclotomic extension of~$\Qp$.
Then~$\Gal_{\Qp(\zeta_{p^{\infty}})}$ is the kernel of the cyclotomic character $\varepsilon:\Gal_{\Qp}\to\Z_p^{\times}$.
By the Ax--Sen--Tate theorem, we have $\Cp^{\Gal_{\Qp(\zeta_{p^{\infty}})}}=\widehat{\Qp(\zeta_{p^{\infty}})}$, the completion of~$\Qp(\zeta_{p^{\infty}})$.
Let~ $V/\Qp$ be a finite-dimensional $\Qp$-vector space equipped with a continuous action of~ $\Gal_{\Qp}$, so that $\rho:\Gal_{\Qp}\to\GL(V)$ is a Galois representation.
Then $V\otimes_{\Qp}\Cp$ has a semilinear action of~$\Gal_{\Qp}$, and Sen showed that it descends to~$\widehat{\Qp(\zeta_{p^{\infty}})}$, in the sense that  \[(V\otimes_{\Qp}\Cp)^{\Gal_{\Qp(\zeta_{p^{\infty}})}}\otimes_{\widehat{\Qp(\zeta_{p^{\infty}})}}\Cp\cong V\otimes_{\Qp}\Cp.\]Furthermore, it descends to $\Qp(\zeta_{p^{\infty}})$: there is a unique $\Gal(\Qp(\zeta_{p^{\infty}})/\Qp)$-stable $\Qp(\zeta_{p^{\infty}})$-subspace $\DSen(V)$ of $V\otimes_{\Qp}\Cp$ such that 
\[\DSen(V)\otimes_{\Qp(\zeta_{p^{\infty}})}\Cp\cong V\otimes_{\Qp}\Cp.\]In fact, $\DSen(V)$ is the union of the finite-dimensional $\Gal(\Qp(\zeta_{p^{\infty}})/\Qp)$-stable $\Qp(\zeta_{p^{\infty}})$-subspaces of $V\otimes_{\Qp}\Cp$.

The $\Qp(\zeta_{p^{\infty}})$-vector space $\DSen(V)$ has an action of $\Gal(\Qp(\zeta_{p^{\infty}})/\Qp)$, and thus a linear action of $\Lie(\Gal(\Qp(\zeta_{p^{\infty}})/\Qp))$.
More explicitly, we have the Sen operator $\Theta_{V}$, which is the $\Qp(\zeta_{p^{\infty}})$-linear map given by $\Theta_{V}\coloneqq \log(\gamma)/\log_p(\varepsilon(\gamma))$ 
for any $\gamma\in \Gal(\Qp(\zeta_{p^{\infty}})/\Qp)$ sufficiently close to~$1$.

By definition, $V$ is Hodge--Tate, with Hodge--Tate weights~$h_1,\dots, h_n\in\Z$, if there exists an isomorphism of $\Cp[\Gal_{\Qp}]$-modules \[V\otimes_{\Qp}\Cp\cong\oplus_{i=1}^{n} \Cp(-h_{i}),\]where~$\Cp(n)$ is the $n$th Tate twist; thus~$V$ is Hodge--Tate if and only if the Sen operator $\Theta_{V}$ is semisimple and has eigenvalues in~$\Z$, in which case the Hodge--Tate weights are the negatives of the eigenvalues of~$\Theta_{V}$.

\subsection{Ordinary Galois representations}\label{subsec:ordinary-Galois-reps}
There are various definitions in the literature of ordinary Galois representations; the following will be convenient for us.
\begin{defn}We say that $\rho:\Gal_{\Qp}\to\GL(V)$ is ordinary if there are integers $h_1\le h_2 \dots\le h_n$  such that~$\rho$ is conjugate to an upper-triangular representation \[
     \begin{pmatrix}
\chi_1 & * & \dots & * \\
0 & \chi_2 & \dots & * \\
\vdots & \vdots & \ddots & \vdots \\
0 & 0 & \dots & \chi_n
\end{pmatrix}
  \]where each $\chi_i:\Gal_{\Qp}\to\Q_{p}^{\times}$ is a character with~$\chi_i|_{I_{\Qp}}=\varepsilon^{-h_{i}}$.
\end{defn}It is straightforward to check that an ordinary Galois representation is de Rham if and only if it is Hodge--Tate, i.e.\ if and only if the Sen operator $\Theta_{V}$ is semisimple.
If $h_1<\dots<h_n$, this is automatic (e.g.\ because the eigenvalues of the Sen operator are distinct), but in general it need not hold.
Indeed, a standard example of a representation which is not Hodge--Tate but whose Sen operator has integral eigenvalues is the ordinary representation \[
  \begin{pmatrix}
    1 & \log_p\varepsilon \\ 0 & 1
  \end{pmatrix},
\]which by definition has Sen operator $
\begin{psmallmatrix}
  0 & 1\\ 0& 0
\end{psmallmatrix}
$.

\section{\texorpdfstring{$p$-adic}{p-adic} modular forms and classicality}\label{sec:p-adic-modular-classicality}
We now  review the approach to $p$-adic modular forms and classicality theorems taken in~\cite{boxer2025modularitytheoremsabeliansurfaces}, which  relies on the work of Pan \cite{MR4390302} and its generalisations to other Shimura varieties by Rodríguez Camargo~\cite{camargo2022locally}, as well as the higher Coleman and Hida theories of Boxer--Pilloni ~\cite{boxer2021higher,boxer2023higher}.
There is a long history of such classicality theorems, going back in the case of Coleman theory to \cite{MR1369416}; we highlight in particular Kassaei's paper \cite{MR2219265}, and its generalisation in~\cite{MR3488741}.
Another important ingredient is the families of $p$-adic automorphic forms introduced by Andreatta--Iovita--Pilloni \cite{MR3275848} (see \cite{MR3966765} for a survey).

For simplicity we say almost nothing below about compactifications of Shimura varieties, or about the distinction between cuspidal and usual cohomology.
These both play an important role in the foundations of the theory, but disappear in our main results, as we always localise at a non-Eisenstein maximal ideal of a Hecke algebra.
Similarly, we will sometimes elide the difference between functors on the abelian and derived level where it makes no difference for our final statements.
We will also only work with $p$-adic modular forms of integral weight, as this is all that is needed for our main theorems and allows us to simplify the exposition in places. We say very little about families of $p$-adically varying weight, although these are an important ingredient in the proofs of some results that we state.

We will also be extremely informal in our treatment of $p$-adic functional analysis and condensed mathematics.  A justification for this is that at the time of writing, the ``correct'' foundations for the constructions we discuss are not yet available. In particular, the arguments below use a $p$-adic version of Beilinson--Bernstein localisation;  ideally,  this  should be defined in the framework of the analytic de Rham stack of Rodríguez Camargo~\cite{camargo2024analyticrhamstackrigid}, similarly to Scholze's
treatment of classical Beilinson--Bernstein localisation  in~\cite{scholzerealgeometrization}. This formalism is expected to be available soon, but in the meantime \cite{boxer2025modularitytheoremsabeliansurfaces} proceeds in a somewhat ad hoc fashion.

While the only Shimura varieties considered here are the Siegel threefolds associated to the group~$\GSp_4 $, many of the results that we explain below are proved in  ~\cite{boxer2025modularitytheoremsabeliansurfaces} for Hodge-type Shimura varieties, or (in the case of results requiring a non-Eisenstein localisation) are expected to hold in this generality. Accordingly, where possible we phrase our results without reference to specific features of~$\GSp_4 $, although we caution the reader that
they should turn to~\cite{boxer2025modularitytheoremsabeliansurfaces} to see the precise hypotheses under which each result is proved.

Accordingly, from now on we write $G=\GSp_4 $; when we eventually need to be concrete, we will realise $G$ as the subgroup of $\mathrm{GL}_{4}$  acting on the free $\Z$-modules of rank $4$, with basis $e_1, \cdots, e_{4}$ and
preserving up to a similitude factor the symplectic form with matrix    $$J=\begin{pmatrix}
        0 & S \\
      -S & 0 
   \end{pmatrix}$$ where $S$ is the $2 \times 2$ anti-diagonal matrix with only $1$'s on the anti-diagonal.        We take $P$ to be
the (``block lower-triangular'') Siegel parabolic stabilising $e_3
,e_4 $, and  $B$ the Borel inside it which is upper-triangular in each of the diagonal
$2\times 2$ blocks.
We let~$M$ be the Levi quotient of~$P$, and let $U,U_P$ be the unipotent radicals of $B,P$ respectively.
We
   let $T$ be the diagonal torus. 
  The sets of $M$-dominant and $G$-dominant characters are respectively  denoted by $X^*(T)^{M,+}$ and $X^*(T)^{+}$. 
We let $\mu  \in X_*(T)$ be the minuscule dominant cocharacter $t\mapsto\mathrm{diag}(1,1,t,t)$.

  We let $W$ be the Weyl group of $G$, with length function
$\ell:W\to\Z_{\ge 0}$, and write~$w_0$ for the longest element of~$W$. Let~$\Phi^+$ be the set of positive roots of~$G$, and let $\rho$ be half the sum of the
positive roots. Write~$W_{M}$ for the Weyl group of~$M$, and let $\WM \subseteq W$ be the set of Kostant representatives of $W_M \backslash W$ (i.e.\
those~$w\in W$ with  $wX^{*}(T)^{+}\subseteq X^{*}(T)^{M,+}$; this is a set of
coset representatives of minimal length).

We denote by $\mathfrak{g}$, $\mathfrak{b}$,
$\mathfrak{h},\mathfrak{p},\mathfrak{m},\mathfrak{u}_{\mathfrak{p}}$ the Lie algebras of $G,B,T,P,M,U_P$ respectively. For each $w \in W$, we let~$P_w\coloneqq w^{-1}Pw$, with Lie algebra $\mathfrak{p}_w\coloneqq w^{-1}
\mathfrak{p} w$, and similarly we define
$\mathfrak{u}_{\mathfrak{p}_w}$, $\mathfrak{m}_w$, and so on. For~$w\in \WM$,  $\mf{b}_{\m_{w}}\coloneqq \mf{b}\cap\mf{m}_{w}$ is a Borel in~$\mf{m}_{w}$.

\subsection{Modular forms and the Hodge--Tate period map}
We fix throughout a tame level~$K^{p}\subset\GSp(\A^{\infty,p})$ (which we ultimately choose in order to guarantee that various spaces of $p$-adic modular forms are 1-dimensional, using the theory of newforms developed in~\cite{MR2344630}), and for each open compact subgroup $K_p\subset\GSp_4 (\Qp)$, we let $\Sh_{K_p}/\Cp$ be the analytic adic space attached to (a toroidal compactification of) the Siegel Shimura variety of level~$K_pK^p$.

For each $\kappa \in X^*(T)^{M,+}$ and each finite level Shimura variety $\Sh_{K_p}$,  we have the usual sheaf~$\omega^\kappa$ of modular forms of weight $\kappa$  on  $\Sh_{K_p}$.
The coherent cohomology $\mathrm{R}\Gamma(\Sh_{K_p}, \omega^\kappa)$ has an action of a Hecke algebra~$\mathbf{T}$ (generated by the Hecke operators at the places where~$K^p$ is hyperspecial), and the corresponding eigenclasses can be computed in terms of automorphic forms on~$G$. More precisely, cuspidal automorphic representations contribute according to their Archimedean components~$\pi_{\infty}$, and the essentially tempered~$\pi_{\infty}$ which contribute to coherent components  are the so-called
non-degenerate limits of discrete series.  The upshot for us is that in order to prove that a Galois representation is modular, it suffices to show that its corresponding system of Hecke eigenvalues contributes to some $\mathrm{R}\Gamma(\Sh_{K_p}, \omega^\kappa)$.

We let $\Sh_{\infty} = \lim_{K_p}\Sh_{K_p}$, which is a perfectoid space with  an action of $\Gal_{\Qp}$, 
admitting a $\Gal_{\Qp}$-equivariant Hodge--Tate period map  \[\pi_{\HT}:\Sh_{\infty}\to \mathcal{FL}, \]where $\mathcal{FL}$ is the base change to~$\C_p$ of the (partial) flag variety $P \backslash G$. This was introduced by Scholze~\cite{scholze-torsion}, and has revolutionised the study of $p$-adic modular forms.
As a first illustration of this, let~$\cL_\kappa$ be the $G$-equivariant sheaf on~$\mathcal{FL}$ whose fibre at~$e$ is the inflation from~$M$ to~$P$ of the irreducible representation of~$M$ of highest weight~$\kappa$, and  we set $\omega^{\kappa,\sm} \coloneqq  (\pi_{\HT}^{*}\mathcal{L}_\kappa)^{\sm}$.
By the definition of~$\pi_{\HT}$, one finds that the sheaf $\omega^{\kappa,\sm}$ descends to~$\omega^\kappa$ on each $\Sh_{K_p}$. Thus $\mathrm{R}\Gamma( \Sh_{\infty}, \omega^{\kappa,\sm})$ is a  complex of smooth admissible  $G(\qq_p)$-representations, equal to $\colim_{K_p} \mathrm{R}\Gamma(\Sh_{K_p}, \omega^\kappa)$.
This suggests the possibility of proving results about modular forms by working on the flag variety~$\mathcal{FL}$, an idea exploited  to great effect by Scholze in~\cite{scholze-torsion}.

Returning to the problem of the modularity of Galois representations, a basic difficulty now presents itself: there is no Galois action on  $\mathrm{R}\Gamma( \Sh_{\infty}, \omega^{\kappa,\sm})$, thus no direct connection to Galois representations.
This difficulty was resolved in the case of the modular curve by Pan~\cite{MR4390302}, who invented geometric Sen theory and combined it with Scholze's ideas to  prove remarkable new ``$p$-adic Eichler--Shimura'' results, relating the $\mathrm{R}\Gamma( \Sh_{\infty}, \omega^{\kappa,\sm})$ to \'etale cohomology groups, which naturally have an action of~$\Gal_{\Q}$.
Furthermore, as we will see below, Pan's theory often allows one to reduce to questions on the flag variety, and thus to explicit computations.

In fact, the spaces of classical modular forms $\mathrm{R}\Gamma( \Sh_{\infty}, \omega^{\kappa,\sm})$ do not directly show up in $p$-adic Eichler--Shimura theory; rather, one sees spaces of $p$-adic modular forms which are defined using the flag variety, and are closely related to the higher Coleman theory of Boxer--Pilloni~\cite{boxer2021higher}.
From our point of view, this is a feature rather than a bug: it is these spaces of $p$-adic modular forms to which we can apply the Taylor--Wiles method in order to prove our modularity lifting theorems, and we then consider separately the problem of proving the classicality of a $p$-adic modular form.
This classicality problem is solved by a generalisation of Pan's ideas, in combination with results of Boxer--Pilloni.

\subsection{Completed cohomology}\label{subsec:completed-cohomology}By Scholze's primitive comparison theorem, the cohomology of the structure sheaf $\mathrm{R}\Gamma_{\an}(\Sh_{\infty}, \oscr_{\Sh_{\infty}})$ is naturally identified with  $\widetilde{\mathrm{R}\Gamma}(\Sh_{\infty}, \qq_p) {\otimes}_{\qq_p} \C_p $, 
where \[\widetilde{\mathrm{R}\Gamma}(\Sh_{\infty}, \qq_p)\coloneqq \lim_{n}\colim_{K_{p}}\mathrm{R}\Gamma(\Sh_{K_p},\Zp/p^{n}) \otimes_{\Zp}\Qp\] denotes Emerton's completed \'etale cohomology (with~$\Qp$ coefficients). As well as the action of $\Gal_{\Qp}$ and of the Hecke algebra~$\mathbf{T}$, this has an action of~$\GSp_4 (\Qp)$, and we can consider the (derived) locally analytic vectors $\widetilde{\mathrm{R}\Gamma}(\Sh_{\infty}, \qq_p)^{\la} {\otimes}_{\qq_p} \C_p $, which have an action of the Lie algebra~$\mathfrak{g}$.
We can define subsheaves of smooth and locally analytic vectors for the $\GSp_4 (\Qp)$-action: 
$$\oscr^{\sm}_{\Sh_{\infty}} \subseteq \oscr^{\la}_{\Sh_{\infty}} \subseteq \oscr_{\Sh_{\infty}},$$ and by
 results of Rodríguez Camargo (following Pan), we have a natural identification
 \begin{equation}
   \label{eq:local-analytic-comparison}
   \widetilde{\mathrm{R}\Gamma}(\Sh_{\infty}, \qq_p)^{\la} {\otimes}_{\qq_p} \C_p=\mathrm{R}\Gamma_{\an}(\Sh_{\infty}, \oscr^{\la}_{\Sh_{\infty}}).
 \end{equation}
 
 Fix~$\lambda\in X^*(T)$,  and write~$M_{\lambda}$ for the Verma module for~$\mf{g}$ with highest weight~$\lambda$.
 By~\eqref{eq:local-analytic-comparison} we have \begin{equation}\label{eqn:weight-nu-CC}\RHom_{\mathfrak{b}}(\lambda, \widetilde{\mathrm{R}\Gamma}(\Sh_{\infty}, \qq_p)^{\la})\otimes_{\Qp}\Cp=\mathrm{R}\Gamma_{\an}(\Sh_{\infty},
   \RHom_{\mathfrak{g}}\bigl(M_{\lambda},\oscr^{\la}_{\Sh_{\infty}})\bigr).\end{equation}We think of the left hand side of~\eqref{eqn:weight-nu-CC} as a space of $p$-adic modular forms of weight~$\lambda$; the goal of $p$-adic Eichler--Shimura theory is to understand it in terms of the $\mathrm{R}\Gamma( \Sh_{\infty}, \omega^{\kappa,\sm})$.

 The formula~\eqref{eqn:weight-nu-CC} motivates the following definitions. Write $\ocal(\mf{g},\mf{b})$ for the usual category~$\cO$ of finitely generated left $U(\mf{g})$-modules with locally finite $\mf{b}$-action and semisimple $\mf{h}$-action,
and   $\ocal(\mf{g},\mf{b})_{\alg}$ for the full subcategory of objects all of whose weights are in~$X^{*}(T)$; this category contains  the Verma modules of integral highest weight. The action of~$\mf{b}$ on an object of $\ocal(\mf{g},\mf{b})_{\alg}$ can naturally be integrated to an action of~$B$, so from now on we regard these objects as $(\mf{g},B)$-modules (and we we regard the left hand side of~\ref{eqn:weight-nu-CC} as a smooth $B(\Qp)$-representation).   Writing~$\Mod_{B(\Qp)}^{\sm}(\oscr^{\sm}_{\Sh_{\infty}})$ for the derived category of solid $\oscr^{\sm}_{\Sh_{\infty}}$-modules with a smooth action of~$B(\Qp)$, the right hand side of~\eqref{eqn:weight-nu-CC} suggests that we should consider the functor   \begin{equation}\label{eqn:cat-O-to-Sh}
\begin{split}
\ocal(\mf{g},\mf{b})_{\alg}&\to\Mod^{\sm}_{B(\Qp)}(\oscr^{\sm}_{\Sh_{\infty}}), \\
M&\mapsto \RHom_{\mathfrak{g}}\bigl(M,\oscr^{\la}_{\Sh_{\infty}}).
\end{split}
\end{equation}
  
\subsection{Localisation to the partial flag variety}
A key point is that the functor $\RHom_{\mathfrak{g}}\bigl(-,\oscr^{\la}_{\Sh_{\infty}})$ factors through a natural analogue of Beilinson--Bernstein localisation for the partial flag variety~$\mathcal{FL}$, as we now explain (see Theorem~\ref{thm:factors-through-localization} below).
From now on we regard~$G$ an an affinoid analytic group over~$\Cp$.
There is a natural action of~$G$ on~$\mathcal{FL}$ by right multiplication, which induces an action of~$\mf{g}$ by derivations on~$\cO_{\mathcal{FL}}$.
We have a filtration of  $G$-equivariant coherent sheaves:
$\mathfrak{u}_{\mf{p}}^0 \subseteq \mathfrak{p}^0 \subseteq \mathfrak{g}^0 = \oscr_{\mathcal{FL}}  \otimes \mathfrak{g}$, whose fibres  at a point $x \in \mathcal{FL} $ are $\mathfrak{u}_{\mf{p}_x} = x^{-1} \mathfrak{u}_{\mf{p}} x \subseteq
\mathfrak{p}_x = x^{-1} \mathfrak{p} x \subseteq 
\mathfrak{g}$.
We let~$G_{r}$ be the analytic subgroup of~$G$ consisting of elements reducing to the identity~$e$ modulo~$p^{r}$, and we set $\cO_{G,e}\coloneqq \colim_r\cO_{G_r}$. Then we define~$\hat{U}(\mf{g})\coloneqq \cO_{G,e}^{\vee}$; this is a completion of the universal enveloping algebra~$U(\mf{g})$.
We also write $G_e=\lim_rG_r$, the limit being taken in the category of locally ringed spaces; this has a single point~$e$ (the identity element of~$G$), with structure sheaf~$\cO_{G,e}$.

We now define the ring of universal twisted differential operators \[\tilde{\mathcal{D}}^{\la} = (\oscr_{\mathcal{FL}} {\otimes} \hat{U}(\mathfrak{g}))/\mathfrak{u}_{\mf{p}}^0(\oscr_{\mathcal{FL}} \otimes \hat{U}(\mathfrak{g})),\] and we let  $\mathcal{C}^{\la}  = ( \oscr_{G,e} \otimes \oscr_{\mathcal{FL}} )^{\mathfrak{u}_{\mf{p}}^0}$ where the invariants are for the diagonal action on the two factors, where the action on  $\oscr_{G,e}$ is by left translation. Then~  $\mathcal{C}^{\la}$ is a $\tilde{\mathcal{D}}^{\la}$-module, and it carries an action of $\mathfrak{g}$ (by right translation on  $\oscr_{G,e}$) which commutes with the $\tilde{\mathcal{D}}^{\la}$-module structure. 
We define a localisation  functor from the (derived) category of solid $U(\mathfrak{g})$-modules to the (derived) category of  solid $\tilde{\mathcal{D}}^{\la}$-modules (i.e.\ twisted $D$-modules):   \begin{eqnarray*}
\mathrm{Loc}: \Mod(U(\mathfrak{g})) &\rightarrow &\Mod(\tilde{\mathcal{D}}^{\la}) \\
M& \mapsto & \RHom_{\mathfrak{g}}(M, \mathcal{C}^{\la}).
\end{eqnarray*}For any~$M$, the sheaf~$\Loc(M)$ admits a natural ``horizontal'' action of the centre $Z(\mf{m})$ of the universal enveloping algebra~$U(\mf{m})$, via the natural injection $\oscr_{\mathcal{FL}} \otimes Z(\mf{m})\into U(\mf{m}^{0})$ (with $\mf{m}^{0}=\mf{g}^0/\mathfrak{u}_{\mf{p}}^0$), and in particular it admits an action of~$\mu\in Z(\mf{m})$. If~$M\in\cO(\mf{g},\mf{b})_{\alg}$ then we regard~$\Loc(M)$ as a~$(\mf{g},B)$-equivariant sheaf on~$\mathcal{FL}$; the
$B$-equivariant structure comes from the action of~ $B$ on~ $M$, while the $\mf{g}$-action comes from the $\tilde{\mathcal{D}}^{\la}$-module structure. 

The following theorem is a consequence of the geometric version of Sen theory introduced by Pan and generalized by Rodríguez Camargo.
Its formulation is a generalisation of Pilloni's interpretation \cite{PilloniVB} of Pan's work for the modular curve.
\begin{thm}  \label{thm:factors-through-localization}The functor~\eqref{eqn:cat-O-to-Sh}   factors through the functor $\Loc$; more precisely, for~$M\in\cO(\mf{g},\mf{b})_{\alg}$ we have \[\RHom_{\mathfrak{g}}\bigl(M,\oscr^{\la}_{\Sh_{\infty}}\bigr)=\bigl(\pi_{\HT}^{*}\Loc(M)\bigr)^{\sm},\]      and
  \begin{equation}
    \label{eq:1}
    \RHom_{\mathfrak{g}}(M, \widetilde{\mathrm{R}\Gamma}(\Sh_{\infty}, \qq_p)^{\la})\otimes_{\Qp}\Cp=\mathrm{R}\Gamma_{\an}(\Sh_{\infty},
   \RHom_{\mathfrak{g}}\bigl(M,\oscr^{\la}_{\Sh_{\infty}})\bigr).
  \end{equation}
Furthermore, the action of $\mu \in Z(\mathfrak{m})$  on $\Loc(M)$ via the
horizontal action induces the Sen operator on the left hand side of~\eqref{eq:1}.
\end{thm}
In view of Theorem~\ref{thm:factors-through-localization}, we will sometimes refer to the horizontal action of $\mu$ as ``the Sen operator'' in the following.
\subsection{\texorpdfstring{$p$}{p}-adic Eichler--Shimura}
\label{subsec:p-adic-eichler}
To go further, we use excision with respect to the Bruhat stratification of~$\mathcal{FL}$,
 i.e.\ the decomposition into $B$-orbits  $\mathcal{FL}=P \backslash G = \coprod_{w \in \WM} P \backslash PwB$,  indexed by the Kostant representatives $\WM$; the dimension of~$P \backslash PwB$ is~$\ell(w)$.
We write $j_w:C_{w} = P\backslash P w B\into \mathcal{FL}$ for
the locally closed immersion of the Bruhat cell corresponding to~$w$, and  $C_{w}^\dag$ for the dagger
neighbourhood of~$C_{w}$ in
$\mathcal{FL}$.
We write $j_{w,\Sh_\infty}:\pi_{\HT}^{-1}(C_w)\into \Sh_{\infty}$ for the morphism induced by~$j_w$.

For each~$w\in\WM$, we can consider the composite of~$\Loc$ with restriction to~$C_{w}^{\dagger}$. Let~$M$ be an object of  $\ocal(\mf{g},\mf{b})_{\alg}$, so that $\Loc(M)$ is a $(\mf{g},B)$-equivariant sheaf; since $C_{w}$ is a $B$-orbit, and~$\Loc(M)$ is $B$-equivariant, the sheaf   $\Loc(M)|_{C_{w}^{\dagger}}$ is determined by its fibre at $w$, which is a representation of the stabiliser of $w$ for the $(\mf{g},B)$-action.
This stabiliser admits an explicit description as follows.
The action of~$(\mf{g},B)$ can be upgraded to an action of the semi-direct product $G_e\rtimes B$, which acts on $C_{w}^{\dagger}$ via $(g,b)\mapsto wgb$.
Write ~$\Stab(w)$ for the stabiliser of~$w$ for this action; then it follows from the definitions that the map $(g,b)\mapsto (gb,b)$ is an injection \[\Stab(w)\into P_w\times B.\] It is then elementary to check that~$\Stab(w)$ is generated by its subgroups $\Stab(w)_e=P_{w,e}\times B_e$ and $P_w\cap B\subseteq B$, and in fact \begin{equation}\label{eqn:explicit-Stab-w}(P_w\cap B)_e\backslash (P_{w,e}\times B_e) \rtimes (P_w\cap B).\end{equation}

The explicit description~\eqref{eqn:explicit-Stab-w} allows us to construct a contravariant functor  \[HCS_{w}:\cO(\m_w,\mf{b}_{\mf{m}_w})_{\alg}\to\Mod_{(\mf{g},B)}(C_{w}^{\dagger})^{\mf{u}_{\mf{p}}^{0}}\] where $\cO(\m_w,\mf{b}_{\mf{m}_w})_{\alg}$ is the algebraic category~$\cO$ for the pair $(\m_w,\mf{b}_{\mf{m}_w})$, as defined above, and the target category has the obvious meaning. 
(Here ``HCS'' stands for ``higher Coleman sheaf''.)
This functor is defined as follows: the action of~$\mf{b}_{\mf{m}_w}$ on~$V\in \cO(\m_w,\mf{b}_{\mf{m}_w})_{\alg}$ integrates to an action of~$B_{M_w}=M_w\cap B$, and thus determines an action of $P_w\cap B$. The admissible $\hat{U}(\mf{g})$-module $(V\otimes_{U(\mf{g})}\hat{U}(\mf{g}))^{\vee}$ therefore has actions of~$P_w\cap B$ and of $P_{w,e}$ (the latter action factoring through the action of~$M_{w,e}$, which comes from the action of~$\mf{m}_w$), and one checks that by allowing~$B_{e}$ to act
trivially, we obtain an action of~$\Stab(w)$.
The corresponding $(\mf{g},B)$-equivariant sheaf is~$HCS_w(V)$.

The reason for introducing this functor is that  there is a commutative diagram 
\begin{equation}\label{eqn:HCS-u-homology}\begin{tikzcd}
	\cO(\mf{g},\mf{b})_{\alg} & \Mod_{(\mf{g},B)}(C_{w}^{\dagger})^{\mf{u}_{\mf{p}}^{0}} \\
	\cO(\m_w,\mf{b}_{\mf{m}_w})_{\alg}
	\arrow["j_w^{-1}\Loc", from=1-1, to=1-2]
	\arrow[        "\Cp\otimes^L_{\mathfrak{u}_{\mf{p}_w}}-"', from=1-1, to=2-1]
	\arrow["HCS_{w}"', from=2-1, to=1-2]
\end{tikzcd}\end{equation}
where the functor $\Cp\otimes^L_{\mathfrak{u}_{\mf{p}_w}}-:\Mod(U(\mathfrak{u}_{\mf{p}_w}))\to \Mod(\Cp)$ is the functor of ``Lie algebra homology'', with $H_i(\mathfrak{u}_{\mf{p}_w},M)=H^{-i}(\Cp\otimes_{U(\mf{g})}M)$.
To see that~\eqref{eqn:HCS-u-homology} commutes, one only has to check that the fibres at~$w$ of the $(\mf{g},B)$-equivariant sheaves are isomorphic representations of~$\Stab(w)$.
This boils down to noting that the fibre  $\mathcal{C}^{\la}_w$ is $\cO_{U_{P_w} \backslash G, e}$, with the action of~$\Stab(w)$ being determined by the action of $P_{w,e}$ by left translation, the action of~$B_e$ by right translation, and the action of $P_w \cap B$ by conjugation.

We now return to the cohomology of Shimura varieties. For any sheaf~$\cF$ on a subset of~$\Sh_\infty$ containing $\pi_{\HT}^{-1}(\cC_{w})$, we write \[\mathrm{R}\Gamma_{w}(\Sh_\infty,\cF)\coloneqq \mathrm{R}\Gamma(\Sh_\infty,j_{w,\Sh,!}\cF|_{\pi_{\HT}^{-1}(\cC_{w})}), \]where $j_{w,\Sh,!} $ is the extension by zero functor on abelian sheaves of solid abelian groups. Then we define the functor
\begin{equation}\label{eqn:HC}
\begin{split}
  HC_w:\ocal(\mf{m}_{w},\mf{b}_{\mf{m}_w})_{\alg}&\to\Mod^{\sm}_{B(\Qp)}(\Cp)
                                                  , \\
M&\mapsto \mathrm{R}\Gamma_{w}(\Sh_\infty,(\pi_{\HT}^{*}HCS_{w}(M))^{\sm}).
\end{split}
\end{equation} 
Combining~\eqref{eqn:HCS-u-homology} with Theorem~\ref{thm:factors-through-localization} and the spectral sequence for a filtered complex, we obtain the following theorem, a basic form of the $p$-adic Eichler--Shimura decomposition.
\begin{thm}\label{thm-p-adic-ES2}  For any $M \in \ocal(\mathfrak{g},\mathfrak{b})_{\alg}$,  we have   a spectral sequence:     $$E_1^{p,q} = \oplus_{w \in \WM, \ell(w)=p} \HH^{p+q}(HC_{w}(\Cp\otimes^L_{\mathfrak{u}_{\mf{p}_w}}M)) $$ converging to 
$\HH^{p+q}(\RHom_{\mathfrak{g}}(M, \widetilde{\mathrm{R}\Gamma}(\Sh_{\infty}, \qq_p)^{\la}))\otimes_{\Qp}\Cp$.   The Sen operator is induced by the action of $w\mu \in Z(\mathfrak{m}_w)$ on
$\HH_*(\mathfrak{u}_{\mf{p}_w},M)$. \end{thm}
\subsection{The ordinary part}\label{subsec:ordinary-ES}Let $T^+(\Qp)\coloneqq \{t\in T(\Qp),\ \forall \alpha\in\Phi^+,\ v_p(\alpha(t))\ge 0\}$.
Given a smooth representation~$V$ of~$B(\Qp)$, there is a Hecke action of $T^+(\Qp)$ on~$V^{U(\Zp)}$; as usual, the action of~$t\in T^+(\Qp)$ is defined via the (normalised) trace $V^{tU(\Zp)t^{-1}}\to V^{U(\Zp)}$. The finite slope part $V^{\nfs}$ of~$V$ is by definition the  subspace on which~$T^+(\Qp)$ acts invertibly.
For any ~$\lambda\in X^{*}(T)_{\R}$, we say that the slopes of~$V^{\nfs}$ are at least~$\lambda$ (respectively, are equal to $ \lambda$) if for every $t\in T^{+}(\Qp)$ and every eigenvalue~$\alpha$ of~$t$ acting on~$V^{\nfs}$, we have $v_p(\alpha)\ge v_p(\lambda(t))$ (respectively, $v_p(\alpha)= v_p(\lambda(t))$). 
The spaces~$HC_{w}(M)^{\nfs}$ are very closely related to the higher Coleman theories of Boxer--Pilloni (see~\cite[Thm.\ 4.6.56]{boxer2025modularitytheoremsabeliansurfaces} for a precise statement). The following key slope bound was proved by Boxer--Pilloni (it is essentially~\cite[Cor.\ 6.2.16]{boxer2023higher}).
It is proved by a careful analysis of integral models of Hecke correspondences (and in particular, it is not deduced from a statement on the flag variety).
We will see below that it immediately implies classicality theorems and $p$-adic Eichler--Shimura decompositions for the ordinary part of the cohomology (and more generally for the ``small slope'' part, although we do not discuss that here).
\begin{thm}\label{thm-bounds-on-slopes}  Let $w \in \WM$, and let $M \in \ocal(\mathfrak{m}_w, \mathfrak{b}_{\mf{m}_w})_{\alg}$ be a module generated by  a highest weight vector of weight $\nu$. 
Then the slopes of $HC_{w}(M)^{\nfs}$  are at least $ - \nu + w^{-1}w_{0,M} \rho + \rho$. 
\end{thm}
We fix from now on a character~$\lambda\in X^{*}(T)^{M,+}$, and assume that $V$ is a smooth representation of~$B(\Qp)$ such that the slopes of~$V^{\nfs}$ are at least~$-\lambda$. Then  the \emph{ordinary part} $V^{\ord}$ of~$V$ is by definition the subspace of~$V^{\nfs}$ whose slopes are equal to~$-\lambda$.

In particular, by Theorem~\ref{thm-bounds-on-slopes}, the slopes of $HC_{w}(M_{\lambda})^{\nfs}$ are all at least~$-\lambda$, so we may consider the ordinary part  $HC_{w}(M_{\lambda})^{\ord}$.
 Write~$d=\dim\mf{u}_{\mf{p}}=\dim \Sh_{K_p}$.
The homology groups $H_{i}(\mf{u}_{\mf{p}_w},M_{\lambda})$ all belong to $\cO(\mf{m}_{w},\mf{b}_{\mf{m}_{w}})$, and they can (at least in principle) be computed by a Chevalley--Eilenberg complex. Write $\Lmw_{\nu}$ for the simple quotient of the Verma module for $\mf{m}_{w}$ of highest weight~$\nu$. Then an analysis of this complex    shows that   $H_{d-\ell(w)}(\mf{u}_{\mf{p}_w},M_{\lambda})$ has a unique subquotient isomorphic to  $\Lmw_{\lambda + w^{-1}w_{0,M}\rho +\rho}$, and that every other Jordan--H\"older factor of any $H_{i}(\mf{u}_{\mf{p}_w},M_{\lambda})$ is generated by a highest weight
vector of the form \begin{equation}\label{eqn:highest-weight-JH-factor}\lambda + w^{-1}w_{0,M}\rho +\rho- \sum_{\alpha \in \Phi^+} n_\alpha \alpha,\end{equation}
 where  $n_\alpha\in\Z_{ \geq 0}$ and  $ n_\alpha >0$ for some~$\alpha$.
Applying Theorem~\ref{thm-bounds-on-slopes} to~\eqref{eqn:highest-weight-JH-factor}, it follows that for any such Jordan--H\"older factor, say~$X$, the slopes of~ $HC_w(X)^{\nfs}$ are at least~ $-\lambda+\sum_{\alpha \in \Phi^+} n_\alpha \alpha>-\lambda$, so that~$HC_{w}(X)^{\ord}=0$.
Going back to Theorem~\ref{thm-p-adic-ES2}, we see 
that the slopes of $\HH^{p+q}(\RHom_{\mathfrak{g}}(M_{\lambda}, \widetilde{\mathrm{R}\Gamma}(\Sh_{\infty}, \qq_p)^{\la}))\otimes_{\Qp}\Cp$ are all at least~ $-\lambda$, and that there is a spectral sequence\begin{equation}\label{eqn:ordinary-ES-spectral-sequence}E_1^{p,q} = \oplus_{w \in \WM, \ell(w)=p} \HH^{2p+q-d}(HC_{w}(L(\mathfrak{m}_w)_{\lambda + w^{-1}w_{0,M}\rho + \rho}))^{\ord}\end{equation} converging to  $\HH^{p+q}(\RHom_{\mathfrak{g}}(M_{\lambda}, \widetilde{\mathrm{R}\Gamma}(\Sh_{\infty},
\qq_p)^{\la}))^{\ord}\otimes_{\Qp}\Cp$.

 Let~$\mf{m}$ be a maximal ideal of~$\mathbf{T}$ with corresponding Galois representation $\rhobar_{\mf{m}}:\Gal_{\Q}\to\GSp_4(\Fpbar)$.
We say that $\mf{m}$ is \emph{non-Eisenstein} if $\rhobar_{\mf{m}}$ is irreducible.
Then a comparison of the cuspidal and non-cuspidal versions of~$HC_{w}$, and an analysis of the cohomology of the boundary of the toroidal compactification (which we have only carried out for~$G=\GSp_4 $, although we expect the analogous results to hold more generally) together show that after localising at a non-Eisenstein~$\mf{m}$, each~$HC_{w}$ only has cohomology in degree~$\ell(w)$, and  $\HH^{i}( \RHom_{\mathfrak{g}}\bigl(M_{\lambda},\oscr^{\la}_{\Sh_{\infty}}))_{\mf{m}}^{\ord}$ vanishes outside of degree~$i=3$ (i.e.\ outside of middle degree). Thus the spectral sequence~\eqref{eqn:ordinary-ES-spectral-sequence} degenerates, proving the following theorem.
\begin{thm}\label{thm-p-adic-filtration-ordinary-nonEis}  Suppose that~$\lambda\in X^{*}(T)$, and that~$\mf{m}$ is non-Eisenstein. Then  \[\RHom_{\mathfrak{b}}(\lambda, \widetilde{\mathrm{R}\Gamma}(\Sh_{\infty}, \qq_p)^{\la})_{\m}^{\ord}\otimes_{\Qp}\Cp=\mathrm{R}\Gamma_{\an}(\Sh_{\infty},
   \RHom_{\mathfrak{g}}\bigl(M_{\lambda},\oscr^{\la}_{\Sh_{\infty}})\bigr)_{\mf{m}}^{\ord}\] 
   is concentrated in degree~$3$, and there is a decreasing filtration on 
$\HH^{3}(\Sh_{\infty},
   \RHom_{\mathfrak{g}}\bigl(M_{\lambda},\oscr^{\la}_{\Sh_{\infty}})\bigr)_{\mf{m}}^{\ord}$ with~$i$th graded piece given by \begin{equation}\label{eqn:graded-piece-L-m-w}\oplus_{w \in \WM, \ell(w)=i} \HH^{i}(HC_{w}(L(\mathfrak{m}_w)_{\lambda + w^{-1}w_{0,M}\rho + \rho}))_{\mf{m}}^{\ord}.\end{equation} \end{thm}

Under appropriate dominance hypotheses, the summands in~\eqref{eqn:graded-piece-L-m-w} can be described in terms of the sheaves~$\omega^{\kappa,\sm}$.
Indeed, recall that ~$\cL_\kappa$ is the $G$-equivariant sheaf on~$\mathcal{FL}$ corresponding to the irreducible representation of~$M$ of highest weight~$\kappa$. Unwinding the definitions, one finds that for each~$w\in\WM$,  we have \begin{equation}\label{eqn:cL-kappa-and-HCSw}\cL_{\kappa}|_{\cC_w^{\dagger}}=HCS_{w}(L(\mathfrak{m_{w}})_{-w^{-1}w_{0,M}\kappa}).\end{equation}
Set \[\kappa_w\coloneqq -w_{0,M}w(\lambda+\rho)-\rho,\] so that \[-w^{-1}w_{0,M}\kappa_w=\lambda+w^{-1}w_{0,M}\rho+\rho.\] Provided that~$\kappa_w\in X^{*}(T)^{M,+}$,  it follows from~\eqref{eqn:cL-kappa-and-HCSw} that \[\omega^{\kappa_{w},\sm}\vert_{ \pi_{HT}^{-1}
    C_w^\dag}=(\pi_{\HT}^{*}HCS_{w}(L(\mathfrak{m}_w)_{\lambda+w^{-1}w_{0,M}\rho+\rho} ))^{\sm},\] and (by~\eqref{eqn:HC}, i.e.\ by the definition of~$HC_{w}$) the contribution from~$w$ to~\eqref{eqn:graded-piece-L-m-w} is \begin{equation}\label{eqn:explicit-L-omega-comparison}\HH^{i}(HC_{w}(L(\mathfrak{m}_w)_{\lambda + w^{-1}w_{0,M}\rho +  \rho}))_{\mf{m}}^{\ord}
  = \HH^{i}_{w}(\Sh_\infty,\omega^{\kappa_{w},\sm})_{\mf{m}}^{\ord}.\end{equation}  Additionally, considering the horizontal action of~$\mu$, we see that the  Sen operator acts on $\omega^{\kappa_{w},\sm}\vert_{ \pi_{HT}^{-1}
    C_w^\dag}$  via $\langle \mu , \kappa_w \rangle$.

We now consider the difference between the spaces of ordinary $p$-adic modular forms $\HH^{i}_{w}(\Sh_{\infty},\omega^{\kappa_w,\sm})_{\mf{m}}^{\ord}$ and of ordinary classical modular forms $\HH^{i}(\Sh_{\infty},\omega^{\kappa_w,\sm})_{\mf{m}}^{\ord}$.
  Bearing in mind~\eqref{eqn:explicit-L-omega-comparison}, the slope bound of Theorem~\ref{thm-bounds-on-slopes} shows that for each pair~$v,w\in\WM$, the slopes of $\HH^{i}_{v}(\Sh_{\infty},\omega^{\kappa_w,\sm})^{\nfs}$ are at least \begin{equation}\label{eqn:slope-bound-Hw-v-stratum}-\lambda+\bigl((\lambda+\rho)-v^{-1}w(\lambda+\rho)\bigr).\end{equation}If we assume that $\lambda+\rho\in X^{*}(T)^{+}$, we see that these slopes are at least~ $-\lambda$, and that $\HH^{i}_{v}(\Sh_{\infty},\omega^{\kappa_w,\sm})_{\mf{m}}^{\ord}$ vanishes unless
  $v^{-1}w(\lambda+\rho)=(\lambda+\rho)$; equivalently, unless $\kappa_v=\kappa_w$.
Accordingly, we now consider the subgroup~$W_{\lambda}$ of~$W$ consisting of those~$w'$ with $w'(\lambda+\rho)=(\lambda+\rho)$. This subgroup corresponds to a standard parabolic subgroup~$Q$ of~$G$; for example, if~$\lambda\in X^{*}(T)^{+}$, then this subgroup is trivial and~$Q=B$.
Then $P\backslash P w Q$ is the union of those Bruhat cells~$C_{v}$ for which $\HH^{i}_{v}(\Sh_{\infty},\omega^{\kappa_w,\sm})_{\mf{m}}^{\ord}$ could be nonzero, and we let $j_{w,Q}: P\backslash P w Q\into \mathcal{FL}$, $j_{w,Q,\Sh_\infty}:\pi_{\HT}^{-1}(P\backslash P w Q)\into \Sh_{\infty}$  be the corresponding locally closed immersions.
Writing \[ \mathrm{R}\Gamma_{w,Q}(\Sh_{\infty},\omega^{\kappa_w,\sm})\coloneqq \ \mathrm{R}\Gamma(\Sh_{\infty},j_{w,Q,\Sh,!}\omega^{\kappa_w,\sm}|_{\pi_{\HT}^{-1}(P\backslash P w Q)}),\]we obtain the following classicality result.
\begin{thm}
  \label{thm:ordinary-classicality-Q-orbit}Suppose that $\kappa_w\in X^{*}(T)^{M,+}$, and $\lambda+\rho\in X^{*}(T)^{+}$.
Then \[ \mathrm{R}\Gamma_{w,Q}(\Sh_{\infty},\omega^{\kappa_w,\sm})_{\mf{m}}^{\ord}= \mathrm{R}\Gamma(\Sh_{\infty},\omega^{\kappa_w,\sm})_{\mf{m}}^{\ord}.\]
\end{thm}
\begin{proof}
It suffices to note that for each $v\notin wW_{\lambda}$, we have $\mathrm{R}\Gamma_{v}(\Sh_{\infty},\omega^{\kappa_w,\sm})_{\mf{m}}^{\ord}=0$ by~\eqref{eqn:slope-bound-Hw-v-stratum}.
\end{proof}
\begin{rem}
  The cohomology   $\mathrm{R}\Gamma_{w,Q}(\Sh_{\infty},\omega^{\kappa_w,\sm})_{\mf{m}}^{\ord}$ is supported in degrees~$[\ell(w_{\min}),\ell(w_{\max})]$, where $w_{\min},w_{\max}$ are respectively the minimal and maximal length representatives in~$\WM$ for the double coset~$W_Mw W_{\lambda}$.
\end{rem}\begin{rem}
  As noted above, if~$\lambda\in X^{*}(T)^{+}$ then~ $Q=B$, so that $\HH^{\ell(w)}_{w,Q}(\Sh_{\infty},\omega^{\kappa_w,\sm})=\HH^{\ell(w)}_{w}(\Sh_{\infty},\omega^{\kappa_w,\sm})$, and the Eichler--Shimura decomposition in Theorem~\ref{thm-p-adic-filtration-ordinary-nonEis} takes a particularly simple form.  However, a stronger result (without taking ordinary parts or making a non-Eisenstein localisation) was already known in this case (going back to Faltings--Chai~\cite{MR1083353} in the Siegel case); from our perspective, it can be
  obtained by replacing~$M_{\lambda}$ above with~$V_{\lambda}$, the algebraic representation of~$G$ of highest weight~$\lambda$.
Then~$\Loc(V_{\lambda})$ is  $G$-equivariant, rather than merely $B$-equivariant, and one can compute directly on the whole of~$\mathcal{FL}$, without needing to restrict to Bruhat strata.
\end{rem}

\subsection{Sen equals Cousin}
\label{sec:sen-equals-cousin}
We now make things more explicit in the case of interest to us, where~$\lambda$ is not regular, i.e.\ $\lambda\not\in X^{*}(T)^{+}$.
We now introduce some more explicit notation for~$G=\GSp_4 $.
We label the elements of $T$ by     $t=\mathrm{diag} (z t_1,  zt_2, zt_{2}^{-1}, zt_1^{-1})$, and the characters $X^*(T)$ of $T$ by tuples
   $\kappa = ( k_1,  k_2; w) \in \Z^2 \times \Z$ with $w \equiv k_1 +k_2 \pmod 2$, where$\kappa (t) =z^w  t_1^{k_1}t_2 ^{k_2 }$. Thus a character $( k_1,  k_2; w)$ is $M$-dominant if $k_1 \geq   k_2$, and $G$-dominant if $0 \geq k_1  \geq k_2$.
   We have~$\rho=(-1,-2;0)$.

   The Weyl group~$W$ is generated by $s_\alpha$ and $s_\beta$ where
$s_\alpha(k_1,k_2;w) = (k_2,k_1;w)$ and $s_\beta(k_1,k_2;w) = (-k_1,k_2;
w)$, so that $w_0=s_\alpha s_\beta s_\alpha s_\beta $, and $w_0 (k_1 ,k_2 ;w)=(-k_1 ,-k_2 ;w)$. We have~$W_M=\{\Id,w_{0,M}=s_\alpha\}$. The elements of $\WM$ are ${^0w}=\Id, {^1w}=s_\beta, {^2w}=s_\beta s_\alpha, {^3w}=s_\beta s_\alpha s_\beta$, where $\ell({^iw})=i$.
From now on we take~$\lambda=(1,1;-2)$, so that $\lambda+\rho=(0,-1;-2)$.
Then~$W_{\lambda}=\{1,s_{\beta}\}$, and~$Q$ is the Klingen parabolic.
We have two $Q$-orbits on~$\mathcal{FL}$, namely $C_{{^0w}}\cup C_{{^1w}}$  and $C_{{^2w}}\cup C_{{^3w}}$, with the corresponding~$\kappa_{w}$ being $(2,2;2)$ and $(1,1;2)$ respectively.
This is illustrated in Figure~\ref{figure:Weyl-chambers-Sp4} (which draws a character $(k_1 ,k_2;w)$ at $(k_1 ,k_2 )$).

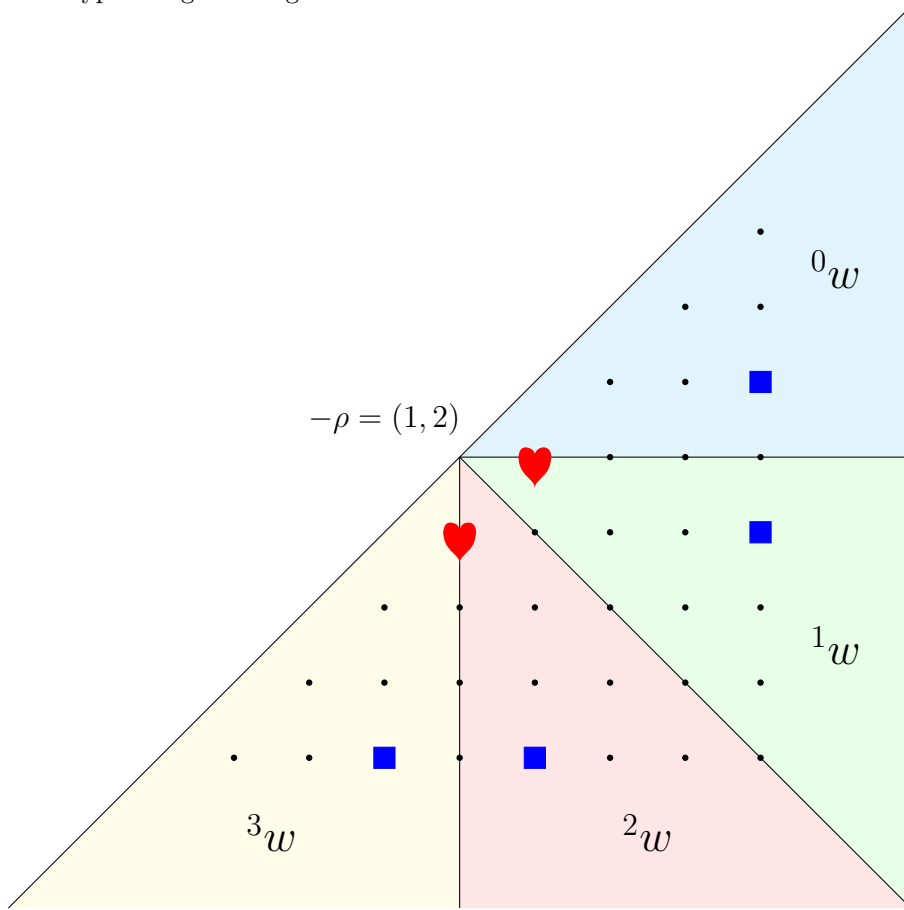
\begin{figure}\caption{(Shifted) $M$-dominant Weyl chambers of weights $(k_1 ,k_2 )$. The $G$-dominant Weyl chamber is labelled by~${}^3w$. The red hearts are at $\kappa_{{}^3w}=\kappa_{{}^2w}=(1,1)$ and $\kappa_{{}^1w}=\kappa_{{}^0w}=(2,2)$ for our~$\lambda=(1,1)$, while the blue squares represent the~$\kappa_w$ for a typical regular weight.}
\label{figure:Weyl-chambers-Sp4}
  \centering
  \begin{tikzpicture}
    
\fill[yellow!10]
    (-5,-4) -- (1,2) -- (1,-4) -- cycle;   
\fill[red!10]
(1,2) -- (1,-4) -- (7,-4) -- cycle;   
\fill[green!10]
    (1,2) -- (7,2) -- (7,-4) -- cycle;   
\fill[cyan!10]
    (7,8) -- (1,2) -- (7,2) -- cycle;   

\foreach \x in {-2,...,5} {
  \pgfmathtruncatemacro{\yMax}{\x}
  \foreach \y in {-2,...,\yMax} {  
    \filldraw[black] (\x,\y) circle (1pt); 
  }
}


\draw (-5,-4) -- (7,8);
\draw (1,2) -- (1,-4);
\draw (1,2) -- (7,2);
\draw (1,2) -- (7,-4);

\node at (0,2.5) {\large $-\rho=(1,2)$};

\draw [red] plot [only marks, mark size=6, mark=heart] coordinates {(1,1) (2,2)};
\draw [blue] plot [only marks, mark size=4, mark=square*] coordinates {(0,-2) (2,-2) (5,1) (5,3)};


\node at (-1.5,-3) {\LARGE ${}^3w$};
\node at (3.5,-3) {\LARGE ${}^2w$};
\node at (6,-0.5) {\LARGE ${}^1w$};
\node at (6,4.5) {\LARGE ${}^0w$};

\end{tikzpicture}
  \end{figure}
In our earlier paper~\cite{BCGP},
the input to our modularity lifting theorem is a class in $H^{0}(\Sh_{\infty},\omega^{(2,2;2),\sm})_{\mf{m}}^{\ord}$, i.e.\ an ordinary Siegel modular form of weight~$2$ (see Remark~\ref{rem:regular weight or not}). Accordingly in that paper we work  with $\mathrm{R}\Gamma_{\id,Q}(\Sh_{\infty},\omega^{(2,2;2),\sm})_{\mf{m}}^{\ord}$, which has cohomology in degrees~$0,1$, and by Theorem~\ref{thm:ordinary-classicality-Q-orbit} agrees with the classical ordinary cohomology  $\mathrm{R}\Gamma(\Sh_{\infty},\omega^{(2,2;2),\sm})_{\mf{m}}^{\ord}$.

However in~\cite{boxer2025modularitytheoremsabeliansurfaces} our input is a class in weight~$3$, rather than weight~$2$.
Using Hida theory, we can produce a congruence to a class in $H^{0}_{\id}(\Sh_{\infty},\omega^{(2,2;2),\sm})_{\mf{m}}^{\ord}$, and ultimately the modularity lifting machinery produces another class in the same cohomology group.  We then need to understand when such a class  extends to $H^0_{\id,Q}(\Sh_{\infty},\omega^{(2,2;2),\sm})_{\mf{m}}^{\ord}$.
We now briefly explain how we do this; in fact for technical reasons we found it more convenient in ~\cite{boxer2025modularitytheoremsabeliansurfaces} to study the problem of extending from $H^2_{{^3w}}(\Sh_{\infty},\omega^{(1,1;2),\sm})_{\mf{m}}^{\ord}$ to $H^2_{{^3w},Q}(\Sh_{\infty},\omega^{(1,1;2),\sm})_{\mf{m}}^{\ord}$, so we do the same here.

From now on we work on $C_{^3w,Q} \coloneqq  C_{^3w} \cup C_{^2w}$. 
We have a short exact sequence of sheaves over $\pi_{HT}^{-1}(C_{^3w,Q})$, corresponding to the stratification of $C_{^3w,Q}$ into $B$-orbits, with $ j_{^3w, \Sh_{\infty}} :  \pi_{HT}^{-1}( C_{^3w}) \hookrightarrow \pi_{HT}^{-1}(C_{^3w,Q})$:
\begin{equation}\label{cous-eq}
 0 \rightarrow  (j_{^3w, \Sh_{\infty}})_! \omega^ {(1,1;2),\sm}\vert_{\pi_{HT}^{-1}(C_{^3w})} \rightarrow  \omega^ {(1,1;2),\sm}\vert_{\pi_{HT}^{-1}(C_{^3w,Q})}   \rightarrow \omega^ {(1,1;2),\sm}\vert_{\pi_{HT}^{-1}(C_{^2w})} \rightarrow 0  
 \end{equation}
Thus  $\mathrm{R}\Gamma_{{^3w},Q}(\Sh_{\infty},\omega^{(1,1;2),\sm})_{\mf{m}}^{\ord}$ is computed by the following complex in degrees $2,3$, where $Cous$ is induced by the class of the extension \eqref{cous-eq}:
\begin{equation}\label{eqn:Cousin-complex} [\HH^2_{^2w}(\Sh_{\infty}, \omega^{(1,1;2),\sm})_{\mf{m}}^{\ord} \xrightarrow{\text{Cous}}
 \HH^3_{^3w}(\Sh_{\infty}, \omega^{ (1,1;2),\sm})_{\mf{m}}^{\ord}].\end{equation}

On the other hand, as noted just below~\eqref{eqn:explicit-L-omega-comparison}, the Sen operator~$\Theta$ on $\HH^{3}( \RHom_{\mathfrak{g}}\bigl(M_{\lambda},\oscr^{\la}_{\Sh_{\infty}}))_{\mf{m}}^{\ord}$ acts by~$\langle\mu,\kappa_{^2w}\rangle=\langle\mu,\kappa_{^3w}\rangle= 0$ on each of $\HH^2_{^2w}(\Sh_{\infty}, \omega^{(1,1;2),\sm})_{\mf{m}}^{\ord}$ and $\HH^3_{^3w}(\Sh_{\infty}, \omega^{ (1,1;2),\sm})_{\mf{m}}^{\ord}$, so it induces  a  map: 
 \begin{equation}\label{eqn:the-Sen-operator-we-will-use}\HH^2_{^2w}(\Sh_{\infty}, \omega^{(1,1;2),\sm})_{\mf{m}}^{\ord} \xrightarrow{\text{Sen}}
 \HH^3_{^3w}(\Sh_{\infty}, \omega^{ (1,1;2),\sm})_{\mf{m}}^{\ord}.\end{equation}
 The key result is the following theorem, showing that ``Sen equals Cousin''.
It is proved by an explicit calculation on the flag variety~$\mathcal{FL}$.
 \begin{thm}\label{thm-Sen-Cousin}   The two maps
   \[\mathrm{Cous}, \mathrm{Sen}:\HH^2_{^2w}(\Sh_{\infty}, \omega^{(1,1;2),\sm})_{\mf{m}}^{\ord} \to
 \HH^3_{^3w}(\Sh_{\infty}, \omega^{ (1,1;2),\sm})_{\mf{m}}^{\ord}\]agree up to a non-zero scalar.                     \end{thm}Combining Theorem~\ref{thm-Sen-Cousin} and~ \eqref{eqn:Cousin-complex}, we see that a class $c\in \HH^2_{^2w}(\Sh_{\infty}, \omega^{(1,1;2),\sm})_{\mf{m}}^{\ord}$ extends to $\HH^{2}_{{^3w},Q}(\Sh_{\infty},\omega^{(1,1;2),\sm})_{\mf{m}}^{\ord}$ (equivalently, to $\HH^{2}(\Sh_{\infty},\omega^{(1,1;2),\sm})_{\mf{m}}^{\ord}$) if and only if~$\Sen(c)=0$.
It only remains to relate this condition to  the Sen operator on the Galois representation associated to a $p$-adic modular form.
We do this using the Eichler--Shimura relations; more precisely, we use the following lemma, which is deduced from results of Nekov\'{a}\v{r},
\cite{MR3942040}.

 \begin{lem}\label{lem-Galois-rep-cc}   Let $f \in \HH^0_{^0w}( \Sh_{\infty}, \omega^{(2,2;2),\sm})^{\ord}$ be an ordinary overconvergent modular eigenform with corresponding Galois representation
$\rho_f: \Gal_{\qq} \rightarrow \mathrm{GSp}_4(\Qpbar)$.  Let $\mathfrak{m}_f$ be the corresponding maximal ideal of the Hecke algebra~$\mathbf{T}[1/p]$. 
We assume that $\rhobar_f$ is irreducible and        the Zariski closure of $\rho_f(\Gal_{\qq})$ contains $\mathrm{Sp}_4$. Then $\HH^3(\mathrm{RHom}_{\mathfrak{b}} ( \lambda, \mathrm{R}\Gamma(\Sh_{\infty}, \Qpbar)^{\la}))^{\ord}[\mathfrak{m}_f] = \rho_f \otimes_{\Qpbar} W$ for some finite-dimensional
vector space $W\neq 0$. \end{lem}

Finally, we deduce from this the following classicality theorem. \begin{thm}\label{thm:multiplicity-one-implies-classical}  Let $f \in \HH^0_{^0w}( \Sh_{K^p}, \omega^{(2,2;2),\sm})^{\ord}$ be an ordinary overconvergent modular eigenform with Galois representation
$\rho_f: \Gal_{\Q} \rightarrow \mathrm{GSp}_4(\Qpbar)$.  Let $\mathfrak{m}_f$
be the corresponding maximal ideal. We assume that:  
\begin{enumerate}
\item the Zariski closure of $\rho_f(\Gal_{\Q})$ contains $\mathrm{Sp}_4$. \item The representation $\rho_f\vert_{\Gal_{\Q_p}}$ is de Rham.
\item\label{ass:dimension-graded-add-up} We have $\mathrm{dim}_{\C_p}
  \HH^i_{^iw}(\Sh_{\infty}, \omega^{(2,2;2),\sm})^{\ord}[\mathfrak{m}_f]
  =1$ for $i=0,1$ and  $\mathrm{dim}_{\C_p}
  \HH^i_{^iw}(\Sh_{\infty}, \omega^{(1,1;2),\sm})^{\ord}[\mathfrak{m}_f]
  =1$ for $i=2,3$.     \item The representation $\rhobar_f$ is irreducible.
\end{enumerate}
Then $f$ is a classical Siegel modular form. 
\end{thm}
\begin{proof}
Assumption~\ref{ass:dimension-graded-add-up} guarantees that the passage to the~$\mf{m}_{f}$ torsion is exact, so that  we have induced maps \begin{equation}\label{eqn:Sen-and-Cous-eigenspace}\mathrm{Cous}, \mathrm{Sen}:\HH^2_{^2w}(\Sh_{\infty}, \omega^{(1,1;2),\sm})^{\ord}[\mf{m}_f] \to
 \HH^3_{^3w}(\Sh_{\infty}, \omega^{ (1,1;2),\sm})^{\ord}[\mf{m}_f].\end{equation}
  These two maps agree by Theorem~\ref{thm-Sen-Cousin}.
Since $\rho_f\vert_{\Gal_{\Q_p}}$ is de Rham, the corresponding Sen operator is semisimple (see Section~\ref{subsec:ordinary-Galois-reps}), which by Lemma~\ref{lem-Galois-rep-cc}  implies that the maps~\eqref{eqn:Sen-and-Cous-eigenspace} vanish.
As noted above, this implies that $\HH^2(\Sh_{\infty}, \omega^{(1,1;2),\sm})^{\ord}[\mf{m}_f]\ne 0$, which by Arthur's multiplicity formula implies that $\HH^0(\Sh_{\infty}, \omega^{(2,2;2),\sm})^{\ord}[\mathfrak{m}_f]\ne 0$, whence it is a one-dimensional $\Cp$-vector space spanned by~$f$, as required.
\end{proof}

\section{Modularity lifting}
\label{sec:tayl-wiles-patch}Finally, we very briefly explain the use of modularity lifting theorems in our two papers.
In~\cite{boxer2025modularitytheoremsabeliansurfaces}, we needed to prove Theorem~\ref{idealthm:modularity lifting}, which we deduce from Theorem~\ref{thm:multiplicity-one-implies-classical}. We work with the ordinary higher Coleman families $\mathrm{R}\Gamma_{^iw}(\Sh_{\infty}, \omega^{\kappa,\sm})^{\ord}$, or rather their integral versions, the higher Hida theories of~\cite{boxer2023higher}.
Since for each~$i$, the complex $\mathrm{R}\Gamma_{^iw}(\Sh_{\infty}, \omega^{\kappa,\sm})^{\ord}$ has cohomology only in degree~$i$, a standard application of the usual Taylor--Wiles method is able to prove modularity lifting theorems for each of these higher Hida theories.
(In practice we find it useful to work with Hida families in which the weight varies, for example in proving local-global compatibility, but this is again standard.)
Furthermore, following Diamond~\cite{MR1440309}, the Taylor--Wiles method is able to prove multiplicity one theorems for spaces of modular forms, and we use it to verify hypothesis~\ref{ass:dimension-graded-add-up} in Theorem~\ref{thm:multiplicity-one-implies-classical}.
(It is here that we use the hypothesis of $p$-distinguishedness in assumption ~\ref{hyp:ordinary-p-distinguished} of Theorem~\ref{idealthm:modularity lifting}.)
The only significant difficulty that we have to overcome in applying the Taylor--Wiles method is that we have to consider residual Galois representations with rather small image, namely~$A_5 $ (when~$p=2$) and~$\GSp_4 (\F_3 )$ (when~$p=3$).
This causes us considerable pain, but does not involve any significant innovations.

As we explained above, in our earlier paper~\cite{BCGP} we worked with the complexes $\mathrm{R}\Gamma_{\id,Q}(\Sh_{\infty},\omega^{(2,2;2),\sm})_{\mf{m}}^{\ord}$, which have cohomology in degrees~$0,1$.
Again, there is an integral version, the higher Hida theory introduced by Pilloni  in~\cite{pilloniHidacomplexes}.
Since there is cohomology in multiple degrees, the usual Taylor--Wiles method does not work in this case, and instead we use the version of the Taylor--Wiles method introduced by Calegari--Geraghty~\cite{CG}.
Our main difficulty now is that rather than working over~$\Q$, we are  over an arbitrary totally real field~$F$ of degree~$d$, say, and the analogous complexes (constructed from the cohomology of Hilbert--Siegel Shimura varieties) now have cohomology in degrees~$0,\dots,d$;
but we do not know how to prove the local-global compatibility results needed for the Calegari--Geraghty method unless~$d=1$. We get around this difficulty by working only with primes~$p$ which split completely in~$F$, and for each~$v|p$, considering complexes which at~$v$ behave like~$\mathrm{R}\Gamma_{\id,Q}$, and at the other places above~$p$ behave like~$\mathrm{R}\Gamma_{\id}$.
We are able to compare these complexes, and we show that an appropriately compatible set of cohomology classes combine to give a classical cohomology class.

\section{Future directions}\label{sec:the-future}

 It is hopefully clear that there is considerable scope for improving on Theorem~\ref{first}.
Firstly, by working with Hilbert--Siegel Shimura varieties, it should be possible to extend Theorem~\ref{first} to abelian surfaces over arbitrary totally real fields.
By making a solvable base change to a totally real field with suitable behaviour at the primes above~$2$, such an improvement would allow us to remove Hypothesis~\ref{conditionattwo} in that theorem, which arose due to the interaction between the local conditions at~$2$ and~$3$ in Lemma~\ref{switching}. Secondly, our  theorems should not require the full strength of the ordinarity condition, but only a suitable ``small slope'' condition, which would allow us to relax  Hypothesis~\ref{conditionatthree} in Theorem~\ref{first}.
Finally, many of the arguments in~\cite{boxer2025modularitytheoremsabeliansurfaces} are written for more general Shimura data, and it would be interesting to prove classicality theorems beyond the case of~$\GSp_4 $.
We intend to pursue all of these ideas in future sequels to~\cite{boxer2025modularitytheoremsabeliansurfaces}. 
Finally, one other possible route  to the modularity of abelian surfaces would be to deduce it from an appropriate version of Serre's conjecture for~$\GSp_4 $. Such a deduction would be  analogous to Khare's theorem~  \cite{MR1434905}, which showed that Serre's conjecture implies the Artin conjecture for odd $2$-dimensional representations of~$\Gal_{\Q}$. With this in mind, we use our theorems to prove the following implication  (see~\cite[Lem.\ 10.4.1]{boxer2025modularitytheoremsabeliansurfaces}). 

\begin{lemma}  \label{higherserre} Suppose that for every residual representation:
$$\rhobar: \Gal_{\Q}  \rightarrow \GSp_4(\F_p)$$
satisfying the following conditions:
\begin{enumerate}
\item $\rhobar$ has multiplier~$\varepsilonbar^{-1}$,
\item   $\rhobar$ is absolutely irreducible, and
\item the semi-simplification of $\rhobar |_{\Gal_{\Qp}}$ is a direct sum of characters,
\end{enumerate}
there exists an ordinary classical  Siegel modular form~$f$ for~$\GSp_4/\Q$ of weight at least~$3$                                  such that \[\rhobar_{f,p}\cong\rhobar.\] Then all abelian surfaces~$A/\Q$ are modular.
\end{lemma}

\section*{Acknowledgements.}
It is a pleasure to thank George Boxer, Frank Calegari, and Vincent Pilloni for their friendship and collaboration, and for their many explanations to me of some of the ideas surveyed here.
I would also like to thank them, together with Matthew Emerton, for their comments on an earlier version of this article.

More generally, my understanding of the mathematics around the Taylor--Wiles method  was shaped over many conversations with (in roughly chronological order) Mark Kisin, Richard Taylor, Tom Barnet-Lamb, David Geraghty, Matthew Emerton, James Newton, George Boxer, Frank Calegari, Jack Thorne and Lue Pan.

Other than the work discussed here, most of my time in the last decade has been spent on my collaborations with Matthew Emerton, and I would like to thank him for the thousands of hours that we have spent discussing mathematics and writing papers together.  

I have been fortunate to collaborate with many other amazing mathematicians who are also wonderful people, and I would like to thank all of my collaborators over the past 20 years.
I owe special thanks to my PhD adviser Kevin Buzzard, and to my postdoctoral advisers Matthew Emerton and Richard Taylor. In particular Kevin spent what is in retrospect an incredible amount of time explaining mathematics to me when I was an undergraduate and postgraduate student.

Finally, in addition to the mathematical debt I owe to them, I would like to thank Kevin Buzzard, Frank Calegari, Ana Caraiani, Matthew Emerton and David Savitt for their constant friendship, support and advice.

This work was supported in part by an ERC Advanced grant and by the Simons Collaboration on Perfection in Algebra, Geometry, and Topology.
This project has received funding from the European Research Council (ERC) under the European Union’s Horizon 2020 research and innovation programme (grant agreement No. 884596).

\bibliographystyle{siamplain}
\bibliography{icm}
\end{document}